\documentclass[10pt,a4paper]{article}
\usepackage{amsfonts}
\usepackage{amscd}
\usepackage{amssymb}
\usepackage{amsthm}
\usepackage{mathtools}
\usepackage{amstext, amsfonts, a4}
\usepackage[T1]{fontenc}
 \usepackage[all,cmtip]{xy}
\usepackage{minitoc}
\input xy
\xyoption{all}

\usepackage[dvipsnames]{xcolor}

\usepackage{kpfonts}
\usepackage{eso-pic}

\usepackage[french,ruled]{algorithm2e}
\usepackage{tikz}
\usetikzlibrary{shapes}
 \usepackage[colorlinks=true]{hyperref}
\hypersetup{urlcolor=blue, citecolor=red}

\usepackage{color,transparent}
\usepackage{graphics,graphicx,subfigure}

\usepackage{comment}
\usepackage{float} 
\def\ogg~{{\rm \og}}   

\def\emptyset{\varnothing}

\def\NN{{\mathbb N}}    
\def\ZZ{{\mathbb Z}}     
\def\QQ{{\mathbb Q}}    

\def\ra{{\rightarrow}}

\def\cA{{\mathcal A}}  \def\cG{{\mathcal G}}  \def\cS{{\mathcal S}} \def\cB{{\mathcal B}}    \def\cT{{\mathcal T}} \def\cC{{\mathcal C}}  \def\cI{{\mathcal I}} \def\cO{{\mathcal O}}  \def\cD{{\mathcal D}}   \def\cP{{\mathcal P}}    \def\cK{{\mathcal K}}   \def\cF{{\mathcal F}}    \def\cX{{\mathcal X}} \def\cY{{\mathcal Y}}  \def\cZ{{\mathcal Z}}

\newcommand{\Inv}{\operatorname{Inv}}

\newcommand{\Ker}{\operatorname{Ker}}
\newcommand{\Id}{\operatorname{Id}}
\def\Im{\operatorname{Im}}                            

\newcommand{\textunderbrace}[2]{%
  \ensuremath{\underbrace{\text{#1}}_{\text{#2}}}%
}

\makeatletter
\def\BState{\State\hskip-\ALG@thistlm}
\makeatother

\def\Hom{{\rm Hom}}
\def\Inf{{\rm Inf}}
\def\Dof{{\rm Def}}

\def\Res{{\rm Res}}
\def\Ind{{\rm Ind}}

\def\Iso{{\rm Iso}}
\def\Mod{{\rm Mod}}
\def\mod{{\rm mod}}

\def\Coker{{\rm Coker}}

\SetKwInput{KwData}{Données}
\SetKwInput{KwResult}{Resultat}
\SetKwInput{KwIn}{Entrées}
\SetKwInput{KwOut}{Sorties}
\SetKw{KwTo}{à}
\SetKw{KwRet}{retourner}
\SetKw{Return}{retourner}
\SetKwBlock{Begin}{debut}{fin}
\SetKwIF{If}{ElseIf}{Else}{si}{alors}{sinon si}{alors}{finsi}
\SetKwSwitch{Switch}{Case}{Other}{suivant}{faire}{cas ou}{autres cas}{fin cas}{fin cas alternatif}
\SetKwFor{For}{pour}{faire}{finpour}
\SetKwFor{While}{tant que}{faire}{fintq}
\SetKwFor{ForEach}{pour chaque}{faire}{finprch}
\SetKwFor{ForAll}{pour tous}{faire}{finprts}
\SetKwRepeat{Repeat}{repeter}{jusqu'à}
\usepackage{fancyhdr}
\fancypagestyle{MonStyle}{%
\fancyhf{}%
\fancyhead[lo]{\nouppercase\rightmark}%
\fancyhead[re]{\nouppercase\leftmark}%
\fancyhead[ro,le]{\thepage}
}
\theoremstyle{plain}
	\newtheorem{Theo}{Theorem}[section] 

\newtheorem{Prop}[Theo]{Proposition}        
\newtheorem{Lemm}[Theo]{Lemma}            
\newtheorem{Coro}[Theo]{Corollary}

\theoremstyle{definition}

        \newtheorem{Defi}[Theo]{Definition}
        
	\newtheorem{Nota}[Theo]{Notation}
        
\theoremstyle{remark}
	\newtheorem{Rema}[Theo]{Remark}

\begin{document}
\author{Ibrahima Tounkara} 
\date{ }
\title{Simplicial Burnside ring}
\maketitle
\begin{center}
\begin{minipage}{13.5cm}{
\setlength{\baselineskip}{.1ex}

\centerline{\small\bf Abstract}

{\small This paper develops links between the Burnside ring of a finite group $G$ and the slice Burnside ring}. The goal is to gain a better understanding of ghost maps,  idempotents, prime spectrum of these Burnside rings and connections between them.}

\end{minipage}
\vspace{2ex}\par

\begin{minipage}{13.5cm}{
\setlength{\baselineskip}{.1ex}
}
\end{minipage}
\end{center}
 ${}$\\
 
\section{Introduction}

Starting from questions in representation theory and homotopy theory, the investigation of  biset functors received considerable attention over the last decades.\\
 We refer to  \cite{S.Bouc}, which covered this subject in detail. We recall some key definitions relevant for the current paper, the presentation follows  recent  articles very closely.
 Let $G$ be a finite group. The {\em Grothendieck ring} constructed from the category of $G$-sets is denoted by $B(G)$ and is called the {\em Burnside ring} of $G$. If $X$ is a finite $G$-set, let $[X]$ be its image in $B(G)$. Additively, $B(G)$ is  the free abelian group on isomorphism classes of transitive $G$-sets. Equivalently, an additive $\ZZ$-basis is given by the $[G/H]$ where $H$ runs through a set  $[C(G)]$ of representatives of conjugacy classes of subgroups of $G$. The multiplication comes from the decomposition of $G/H \times G/K$ into orbits. The ring $B(G)$ is commutative with unit $[G/G]$.  If $H$ is a subgroup  of $G$, then there is a unique linear form $\phi_H: B(G )\ra \ZZ$ such that $\phi_H([X])= \vert X^H \vert $ for any $G$-set  $X$. It is clear moreover that $\phi_H$ is a  ring homomorphism, and  {\em Burnside's theorem} ( \cite{Burnside} Chap. XII Theorem I) is equivalent to the {\em ghost map}
 $\Phi = \prod_{H \in [C(G)]} \phi^G_H : B(G) \ra \prod_{H \in [C(G)]} \ZZ$
 being injective.
 The cokernel of the ghost map is finite, and has been explicitly described by Dress \cite{Dress}. In particular,  the ghost map $\QQ \Phi:  \QQ B(G) \ra \prod_{H \in [C(G)]} \QQ$ is an algebra isomorphism, where $\QQ B(G)= \QQ \otimes_{\ZZ} B(G)$. This shows that $\QQ B(G)$ is a split semisimple commutative $\QQ$-algebra.  Explicit formulas for its primitive idempotents have been given by Gluck \cite{Gluck} and independently by Yoshida \cite{Yoshida}.
 Andreas Dress proved in \cite{Dress} that  if $p$ is  $0$ or a prime and $I_{H,p}= \{ X \in B(G)~\vert ~ \phi_H( X)  \in p\ZZ\}$, then any prime ideal in $B(G)$ is of the form $I_{H,p}$ for some $H,p$. Moreover, a finite group is solvable if and only if the spectrum of its Burnside ring is connected (in the sense of Zariski's topology), i.e.  if and only if $0$ and $1$ are the only idempotents in $B(G)$. 
If $X$ is a finite $G$-set, the $\QQ$-vector space $\QQ X$ with basis $X$ has a natural $\QQ G$-module structure, induced by the action of $G$ on $X$. The construction $X \mapsto \QQ X$ maps disjoint unions of $G$-sets to direct sums of $\QQ G$-modules, and so it induces a map $Ch: B(G) \ra R_\QQ(G)$. There are important connections between  the Burnside ring and the permutation representations.
This latter map, leads to an associated map $Spec:  Spec (R_\QQ(G)) \ra Spec B(G)$ which is always injective (See \cite{Dress1}).\par
The {\em slice Burnside ring} $\Xi(G)$ introduced by Serge Bouc,  is built as the Grothendieck ring of the category of {\em morphisms} of finite $G$-sets, instead of the category of finite $G$-sets used to build the usual Burnside ring. It  shares  almost  all   properties of the Burnside ring. In particular, as already shown (see \cite{S.Bouc1} for a more  complete description), the slice Burnside ring  is a commutative ring, which is free of finite rank as a $\ZZ$-module. 
The investigation of the {\em slices}, that is the pairs of groups $(T,S)$ such that $S$ is a subgroup of $T$, is a central subject in the study of the slice Burnside ring. One reason for considering slices is that if $f: X \ra Y$ is a morphism of finite $G$-sets  then in $\Xi(G)$
$$[ \xymatrix{X \ar[r]^{f} &  Y} ]= \sum_{x \in [G\backslash X]} [ \xymatrix{G/G_x \ar[r]^{p} &  G/G_{f(x)}} ],$$
where square brackets to denote here the image of the isomorphism class of $f$ in $\Xi(G)$ and $G_x$ denotes the stabilizer of $x$. 
Thus, the group $\Xi(G)$ is generated by the elements $[ \xymatrix{G/S_0 \ar[r]^{p} &  G/S_1} ]$ where $(S_1,S_0) $ runs through a set $[\Pi(G)]$ of representatives of conjugacy classes of slices of $G$.   One  can show that this generating set is actually a basis of the slice Burnside group.
There is an analogue of Burnside's theorem. After  tensoring with $\QQ$, the slice Burnside ring becomes a split semisimple $\QQ$-algebra, and an explicit formula for his primitive idempotents  can be stated. The prime spectrum of this ring has been described, and Dress's characterization of solvable groups in terms of the connectedness of the spectrum of the Burnside ring can be generalized as well.\par
In this paper, we introduce a ring $B_n(G)$, for $n \in \NN$ such that $B_0(G)=B(G)$ and $B_1(G)=\Xi(G)$. We extend to $B_n(G)$ most of the properties recalled above in the case $n=0$ or $n=1$.\par
Recall that any poset $\Pi$ can be treated as a category in which the objects are the elements of $\Pi$ and in which there is exactly one morphism $x \ra y$ if $x \le y$ and there are no other morphisms. The nerve of $\Pi$ is then the same as the ordered  simplicial complex associated ( that is the vertices of $\Pi_n$  are the objects of $\Pi$ and the   $n$-simplices are the chains of objects of $\Pi$ of length $n$, with face maps given by $d_i(S_0, \dots, S_n)= (S_0, \dots, \hat S_i , \dots, S_n)$, where as usual, the term $ \hat {} $ denotes a term that is being omitted and the degeneracy maps given by   $s_j(S_0, \dots, S_n)= (S_0, \dots, S_j, S_j , \dots, S_n)$ ). In particular, if we consider the collection of all subgroups of $G$ ordered by inclusion, we get the nerve category  $\Pi_\bullet(G)$ where elements in $\Pi_n(G)$, whose are just chains 
$$\bar \cS:S_0 \subseteq S_{1} \subseteq  \dots \subseteq  S_n$$
 will be called  {\em $n$-slice}, this is the simplicial complex considered in \cite{Quillen} for $p$-groups. Our  ring $B_n(G)$ has basis the set of conjugacy classes of $n$-slices, with multiplication given by
$$(S_0, \dots, S_n)\cdot (T_0, \dots, T_n)=  \sum\limits_{\substack{g \in [S_0\backslash G/ T_0]}}  (S_0 \cap {}^{g}T_0, \dots,   S_n \cap {}^{g}T_n ).$$ 
 This paper is organized as follows:\\
In Sect. 2, we examine the Grothendieck group of the nerve of the skeleton   of the category  of finite $G$-sets (we abuse for using  the term nerve of  ${}_GSet$). It turns out that the obtained group $B_n(G)$, called the $n$-simplicial group,  is very similar to the classical Burnside group. It is worthwhile to discover wether these well-known properties of the Burnside ring  characterize $B_n(G)$, since then will known regard this ring as a "geometric realization" of some simplicial $G$-set.     
Hence in this part we  deepen the links between the classical Burnside rings and the slice Burnside rings. In Sect. 3, we establish that the $n$-simplicial  Burnside ring embeds in a product of copies of the integers, via a {\em ghost map}, and this map has a finite cokernel. It turns out that $B_n(G)$ is a commutative  semisimple  algebra afer tensoring with $\QQ$, isomorphic to a direct sum indexed by  $[\Pi_n(G)]$ of copies of $\QQ$. In sect. 4. we give an explicit formula for the  primitive idempotents of $\QQ B_n(G)$. Sect.5 is devoted to the study of the prime spectrum of $B_n(G)$ by extending  Dress's characterization of solvable groups for $B(G)$. The last section examines the Green biset functor structure of $B_n$.




\section{Simplicial Burnside group}
 For $G$  a finite group, let $\cC_\bullet^G$ be the {\em nerve  of the category} ${}_GSet$ of finite $G$-sets (see \cite{May} for more details, \cite{Maclane} P.177), that is,  the simplicial set whose $n$-simplices are diagrams 
 $$\cC_n^G= \{\cX_n^f:= \xymatrix{X_0 \ar[r]^{f_1} &  X_{1}\ar[r]^-{f_{2}} & \dots X_{n-1} \ar[r]^-{f_{n}} &  X_{n}} \}$$
 where the $X_i$ are $G-sets$ and the $f_i$ are morphisms of $G$-sets, and the {\em degeneracy} $s_i$ and {\em face} $d_i$ \footnote{Normally, we might be careful to label the face maps from $\cC_n^G$ to $\cC_{n-1}^G$ as $d^n_0, \dots, d^n_n$, similarly for the degeneracy maps, but this is rarely done in practice.} maps are defined by including an identity $\xymatrix{X_i \ar[r]^{id} & X_i}$, and leaving out $X_0$ if $i=0$, contracting 
$\xymatrix{X_{i-1} \ar[r]^{f_i} & X_{i}\ar[r]^{f_{i+1}} & X_{i+1}}$ to  $\xymatrix{X_{i-1} \ar[r]^{f_{i+1} \circ f_i} & X_{i+1}}$ if $0<i<n$, leaving $X_n$ if $i=n$, respectively. 
Then the following identities may be verified directly
\begin{equation}\label{1}
d_j \circ d_i  = d_{i}  \circ d_{j+1} ~~~ \mbox{for} ~~  i \le j \\
\end{equation}
\begin{equation}\label{2}
 s_i \circ s_j  = s_{j+1} \circ  s_{i}  ~~~ \mbox{for} ~~  i \le j\\
\end{equation}
\begin{equation}\label{3}
 d_i \circ  s_{j} = \left\{ \begin{array}{rcl}
 s_{j-1}  \circ  d_{i}& \mbox{if}&  i < j,\\
 id_{[n]} & \mbox{if}& i= j, j+1,\\
 s_{j} \circ  d_{i-1}  & \mbox{otherwise }&.
\end{array}\right.
\end{equation}

\begin{Nota} \label{operators} 
Let $\cC_n^G$ be the category defined as follows:
\begin{itemize}
\item The {\em objects} of $\cC_n^G$ are  sequences $\cX^f_n: \xymatrix{X_0 \ar[r]^{f_1} &  X_{1}\ar[r]^-{f_{2}} & \dots X_{n-1} \ar[r]^-{f_{n}} &  X_{n}}$ of morphisms of $G$-sets called $(n,G)$-simplices. 
\item If $\cX^f_n$ and $\cY^g_n$ are objects of  $\cC_n^G$, a {\em morphism} from $\cX^f_n$ to $\cY^g_n$ is  a family $(\mu_i: X_i \ra Y_i)_{0 \le i \le n}$ of morphisms of $G$-sets such that $\mu_{i} \circ f_i = g_i \circ \mu_{i-1}$, for $i=1, \dots, n$.
\item Morphisms of $(n,G)$-simplices   compose in the obvious way (coordinate-wise). 
\end{itemize}
\end{Nota}





\begin{Prop}
 For a non-negative integer $n$, the category  $\cC_n^G$  has finite products $\times$ and  coproducts  $\sqcup$  induced by those of the category of finite  $G$-sets, respectively. It has also an  initial object $\emptyset=  (\xymatrix{\emptyset \ar[r]^{} &  \emptyset \ar[r]^-{} & \dots \emptyset \ar[r]^-{} &  \emptyset })$.
\end{Prop}
\begin{proof}
This is straightforward.
\end{proof}

\begin{Defi}
Let $G$ be a finite group. 
  A {\em $n$-slice}  of $G$ is a $(n+1)$-tuple $(S_0, S_{1}, \dots, S_n)$ of subgroups of $G$, with $S_{i-1} \le S_i$,  $\forall i   \in \{1, \dots, n\}$. It is helpful to refer the tuple as $\bar \cS$, in which each composite $ \cS_i$ as representing the group $S_i$. The set of all $n$-slices of $G$ will be denoted by $\Pi_n(G)$.
\end{Defi}

\begin{Defi} For any nerve $\cC_\bullet$, define a pre-ordering $\preceq$ of $Ob(\cC_n)$ by $A \preceq B$ if $\Hom_{\cC_n} (A, B) \neq \emptyset$ and an equivalence $\cong$ on $Ob(\cC_n)$ by $A \cong B$ if and only if $A \preceq B$ and $B \preceq A$. So, on $\Pi_n(G)$, we have  the following relation
$$\bar \cT:= (\cT_0, \dots, \cT_n) \preceq \bar \cS:= (\cS_0, \dots, \cS_n) \Longleftrightarrow \cT_i \le \cS_i,~~~~~~ \forall i=0,\dots, n$$
\end{Defi}


  Recall that in ${}_GSet$, $X$ is indecomposable if and only if $X$ is simple and any simple $G$-set is isomorphic to $G/H$ for some subgroup $H$ of $G$. 
 \begin{Nota}${}$
 \begin{itemize}
\item Any  $n$-slice $\bar \cS$ of $G$ gives rise to an $(n,G)$-simplex
$$(G/\bar \cS)_n := (\xymatrix{G/S_0 \ar[r]^{p_1} &  G/S_1\ar[r]^{p_2} & \dots G/S_{n-1} \ar[r]^{p_{n}} &  G/S_n}),$$ 
where $p_i$ are the projection morphisms.\\ 
The $(n,G)$-simplices   $(G/\bar \cS)_n$ are {\em indecomposable} in the sense that  $(G/\bar \cS)_n=\cY_n^g \sqcup ~\cZ_n^h \Longrightarrow \cY_n^g=\emptyset$ or $\cZ_n^h=\emptyset$.
\item  For  a $(n,G)$-simplex $\cX^f_n$, we set
$$\phi_{ \bar \cS} \big(\cX^f_n \big):= \big\vert \Hom_{\cC^G_n} \big((G/ \bar \cS)_n,  \cX^f_n \big) \big\vert,$$ the number of elements in the set $\Hom_{\cC^G_n} \big((G/ \bar \cS)_n,  \cX^f_n \big)$.
\end{itemize}
\end{Nota}

\begin{Prop}  Let  $\cX^f_n$, $\cY_n^g$ be $(n,G)$-simplices 
$$\phi_{\bar \cS} (\cX^f_n)=\Bigg\vert f_1^{-1}\Bigg(  f_2^{-1}\Bigg( ... f_i^{-1}\bigg(...f_{n-1}^{-1}\big( f_{n}^{-1}(X_n^{S_n}) ^{S_{n-1}}\big)^{S_{n-2}} ...\bigg)^{S_i} ...  \Bigg)^{S_1} \Bigg)^{S_0}\Bigg\vert, $$
in particular, for any $n$-slice $\bar \cT$, one has
$$\phi_{\bar \cS}((G/ \bar \cT)_n)= \vert \{ g \in G/\cT_0 ~\vert~ {\bar \cS}^g \preceq \bar \cT\} \vert.$$
\end{Prop} 

\begin{proof}${}$
 Observe that for any $i= 0, \dots, n$, there is a bijection between $\Hom(G/\cS_i, X_i)$ and the set $X_i^{\cS_i}:=\{x \in X_i ~\vert~ gx=x~~\text{for all} ~g \in \cS_i\}$. Indeed, each $f$  in $ \Hom(G/\cS_i, X_i)$ maps the element $\cS_i \in G/\cS_i$ onto an element $x_i \in X_i$ which is invariant  by $\cS_i$. Further, $f$ is completely  determined by $x_i$, since $f(g \cS_i)=gx_i$ for all $g \in G$. The correspondence $f\mapsto x_i$ gives the desired bijection. 
Therefore, the set $\Hom_{} \big((G/ \bar \cS)_n,  \cX^f_n \big)$  is in bijection with de set $$\{(x_0,x_1, \!\dots,\! x_n)\! \in X_0^{S_0}\! \times \!X_1^{S_1} \!\times\! \dots \!\times X_n^{S_n} \vert ~  f_1(x_0)\!=\!x_1;  \dots; f_{n}(x_{n-1})\!=\!x_n \},$$
where the last equalities follow from the commutativity of the diagram
\[
 \xymatrix{
 G/S_0 \ar[d]_{\mu_0} \ar[r]^{p_1}&G/S_{1} \ar[d]_{\mu_{1}} \ar[r]^{p_{2}}&G/S_{2}  \ar[d]_{\mu_{2}} \ar[r]& \dots \dots &G/S_{n-1}\ar[d]_{\mu_{n-1}} \ar[r]^{p_{n}}&G/S_n\ar[d]^{\mu_n}\\
 X_0\ar[r]^{f_{1}}&X_{1}\ar[r]^{f_{2}}&X_{n-2} \ar[r]&  \dots \dots& X_{n-1}\ar[r]^{f_{n}}&X_n
 }
\]
If we take $X_i=G/\cT_i$ for $i=0,\dots, n$, we note that the left coset $g\cT_i$ is $\cS_i$-invariant if and only if $\cS_i.g \cT_i=g\cT_i$, that is, $g^{-1}\cS_ig \le \cT_i$. This establishes that $$(G/ \cT_i)^{\cS_i}=\{g \cT_i~\vert~ g\in G, g^{-1}S_ig \le \cT_i \}$$ and therefore  $(G/\cT_i)^{S_i}= \emptyset $ unless $\cS_i \le_G \cT_i$ (  by which we mean  that   $S_i$ is conjugated to a subgroup of $\cT_i$ ). In particular,
\begin{eqnarray*}
\phi_{\bar \cS}((G/ \bar \cT)_n) &=& \vert \{(g\cT_0, \dots, g\cT_n) ~\vert~ g \in G, S_0^g\le \cT_0, \dots, S_n^g\le \cT_n ~\text{and}~g\cT_{i-1}= g\cT_{i}  \} \vert\\
    &=&\vert \{ g \in G/\cT_0 ~\vert~ {\bar \cS}^g \preceq \bar \cT\} \vert.
\end{eqnarray*}

\end{proof}

\begin{Coro} \label{conj}${}$  Let  $\cX^f_n$, $\cY_n^g$ be $(n,G)$-simplices 
\begin{enumerate}
\item If $\cX^f_n$ and  $\cY_n^g$ are isomorphic then $\phi_{\bar \cS}(\cX_n^f ) =\phi_{\bar \cS}(\cY_n^g )$. 
\item \label{congrue}
Let $p$ be a prime and let $\bar \cS$ be a $n$-slice of $G$. If $P$ is a $p$-subgroup of $N_G(\bar \cS)=\cap_{i=0}^n N_G(S_i)$ and $P\bar \cS$ denotes the $n$-slice $(PS_0, \dots, PS_n)$, then 
$$\phi_{\bar \cS}(\cX^f_n) \equiv \phi_{P\bar \cS}(\cX^f_n)~ (\mod.~p),$$
for any $(n,G)$-simplex $\cX^f_n$.

\item Two indecomposable $(n,G)$-simplices $ (G/ \bar \cS)_n$ and  $(G/ \bar \cT)_n$ are isomorphic if and only if the $n$-slices $ \bar \cS$ and $\bar \cT$ are conjugate (we set $\bar \cS=_G \bar \cT$).
\item If $\bar \cS$ and $\bar \cT$ are two $n$-slices, then $\phi_{\bar \cS}(\cX_n^f ) \le \phi_{\bar \cT}(\cX_n^f )$ for any $(n,G)$-simplex $\cX_n^f$ if and only if $\bar \cT \preceq_G \bar \cS$.\\ In particular, $\phi_{\bar \cS}(\cX_n^f ) = \phi_{\bar \cT}(\cX_n^f )$ for all $(n,G)$-simplices $\cX_n^f$ if and only if $\bar \cT=_G~ \bar \cS$.

\item For any two indecomposable $(n,G)$-simplices $(G/\bar \cS)_n$ and  $(G/\bar \cT)_n$, one has    $\phi_{\bar \cS}((G/ \bar \cS)_n)$ divides $\phi_{\bar \cT}((G/ \bar \cS)_n)$.

\end{enumerate}
   
 \end{Coro}  
 
 \begin{proof}${}$
\begin{enumerate}
\item It is clear, since any  isomorphism  of $G$-sets $\mu_i: X_i \ra Y_i$ induces a bijection $X_i^S\ra Y_i^S$ on the sets of fixed points by any subgroup $S$ of $G$.
\item For any $(n,G)$-simplex $\cX_n^f$, the set $$ \Inv_{\bar \cS} (\cX_n^f):=  f_1^{-1}\Bigg(  f_2^{-1}\Bigg( ... f_i^{-1}\bigg(...f_{n-1}^{-1}\big( f_{n}^{-1}(X_n^{S_n}) ^{S_{n-1}}\big)^{S_{n-2}} ...\bigg)^{S_i} ...  \Bigg)^{S_1} \Bigg)^{S_0}$$
is invariant by $N_G(\bar \cS )$, and so 
$$\vert  \Inv_{\bar \cS} (\cX_n^f) \vert  \equiv \Bigg\vert f_1^{-1}\Bigg(   ... f_i^{-1}\bigg(... f_{n}^{-1}(X_n^{S_n}) ^{S_{n-1}} ...\bigg)^{S_i} ...  \Bigg)^{PS_0}  \Bigg\vert~(\mod~p),$$
and moreover 
$$f_1^{-1}\Bigg(   ... f_i^{-1}\bigg(... f_{n}^{-1}(X_n^{S_n}) ^{S_{n-1}} ...\bigg)^{S_i} ...  \Bigg)^{PS_0} \equiv f_1^{-1}\Bigg(   ... f_i^{-1}\bigg(... f_{n}^{-1}(X_n^{PS_n}) ^{PS_{n-1}} ...\bigg)^{PS_i} ...  \Bigg)^{PS_0}.$$
 \item If $(G/\bar \cS)_n \cong (G/\bar \cT)_n$ then $\Hom(G/\bar \cS, G/\bar \cT) \neq \emptyset $ and so $\bar \cS \le_G  \bar \cT$. Therefore $ \bar \cS =_G \bar \cT$ by symmetry. Conversely  if $\bar \cS =_G \bar \cT$, for example $\bar \cS^{g} = \bar \cT$ with $g \in G$, then there exists an isomorphism of $(n,G)$-simplices $(\mu_i: G/\cT_i \ra G/\cS_i)_{0 \le i \le n}$, given by
$$\mu_i (g'T_i) =\mu_i (g'g^{-1}S_ig)=g'g^{-1}S_i, ~~\text{for}~ g' \in G.$$
So $(G/\bar \cS)_n \cong (G/\bar \cT)_n$ as $(n,G)$-simplices if and only $\bar \cS=_G\bar \cT$.

\item If $\phi_{\bar \cS}(\cX_n^f ) \le \phi_{\bar \cT}(\cX_n^f )$ for any $(n,G)$-simplex, then in particular $ \phi_{\bar \cT}((G/ \bar \cS)_n ) \neq 0$ since $ \phi_{\bar \cS}(G/ \bar \cS ) \neq 0$, and so $\bar \cT \preceq_G \bar \cS$. On the other hand, if $ \bar \cT \preceq {}^g  \bar \cS= \bar \cK $, then $\phi_{\bar \cS}(\cX_n^f )= \vert  \Inv_{\bar \cS} (\cX_n^f) \vert =     \vert  \Inv_{\bar \cK} (\cX_n^f)  \vert \le \vert  \Inv_{\bar \cT} (\cX_n^f)   \vert = \phi_{\bar \cT}(\cX_n^f )$.
\item Consider the action of $N_G(\bar \cS)$ on $G/S_0$ defined by  $$x.gS_0= gx^{-1}S_0$$
for $x\in N_G(\bar \cS)$ and $gS_0 \in G/S_0$. Then $S_0$ acts trivially on $G/S_0$ and so it becomes a left $N_G(\bar \cS)/S_0$. Moreover, $N_G(\bar \cS)/S_0$ acts freely on $G/S_0$. Note that for any $(n,G)$-slice  $\bar \cT$,  the set 
 $$\Inv_{\bar \cT} ((G/\bar \cS)_n):=\{gS_0 ~\vert ~ \bar \cT\cdot g \bar \cS= g\bar \cS \} $$  is an $N_G(\bar \cS)/S_0$-subset of $G/S_0$. So $N_G(\bar \cS)/S_0$ acts freely on  $\Inv_{\bar \cT} ((G/\bar \cS)_n)$, and so $\vert N_G(\bar \cS)/S_0 \vert $ divides $\vert  \Inv_{\bar \cT} ((G/\bar \cS)_n) \vert$.

\end{enumerate}
\end{proof}

\begin{Defi}
A $(n,G)$-simplex  $\cX^f_n$ is called {\em $i$-sliceable}  with $i$ in $\{0,\dots, n\}$ if $X_i= A_i \sqcup B_i$ as disjoint union of two non-empty $G$-sets.\\
\end{Defi}

\begin{Rema}
Note that for any {\em i-sliceable} $(n,G)$-simplex  $\cX^f_n$,  one has $(n,G)$-simplices  
\begin{itemize}
\item $\cA_n^f:\xymatrix{A_0 \ar[r]^{f_1} &  A_{1}\ar[r]^{f_{2}} &  \ar[r] \dots A_{i} \ar[r]^{f_{i+1}} & X_{i+1} \dots \ar[r]^{f_{n-1}}   &X_{n-1} \ar[r]^{f_{n}} &  X_{n}} $, 
\item $\cB_n^f: \xymatrix{B_0 \ar[r]^{f_1} &  B_{1}\ar[r]^{f_{1}} &  \ar[r] \dots B_i \ar[r]^{f_{i+1}} & X_{i+1} \dots \ar[r]^{f_{n-1}}  &X_{n-1} \ar[r]^{f_{n}} &  X_{n}}$,
\end{itemize}
defined inductively by 
$$A_{j-1}:= f_{j }^{-1}(A_j) ~~\text{and}~~ B_{j-1}:= f_{j }^{-1}(B_j),$$ for any $1 \le j\le i$.\\
We set $[\cX^f_n, \cA^f_n, \cB^f_n]_i$ to denote the corresponding triple. If $\Gamma$ denotes the class of such triple $(\cX^f_n, \cA^f_n, \cB^f_n)$, then $\Gamma$ is closed by isomorphism.
\end{Rema}

\begin{Prop}\label{defining}${}$\\
For any $(\cX_n^f, \cA_n^f, \cB_n^f) \in \Gamma$ and for a $n$-slice $\bar \cS$ fixed, we  have
\begin{equation}
   \phi_{\bar \cS} (\cX^f_n)= \phi_{\bar \cS} (\cA^f_n) + \phi_{\bar \cS} (\cB^f_n).
   \end{equation}
\end{Prop}
\begin{proof}
Let $x_n$ be an element of $X_{n}^{S_{n}}$ and $i$ be an integer such that $X_i= A_i \sqcup B_i$.Then, either $f_{i+1}^{-1} \circ \dots \circ f_{n}^{-1}(x_n) \in A_i$ and $f_{j+1}^{-1} \circ \dots \circ f_{n}^{-1}(x_n) \in A_j$ for any $j<i$,\\
 or $f_{i+1}^{-1} \circ \dots \circ f_{n}^{-1}(x_n) \in B_i$ and $f_{j+1}^{-1} \circ \dots \circ f_{n}^{-1}(x_n) \in B_j$ for any $j<i$.
 Hence $\phi_{\bar \cS} (\cX^f_n)= \vert \big(A_0\cap f^{-1}(\cX_n^f)^{\bar \cS} \big)\sqcup \big( B_0\cap f^{-1}(\cX_n^f)^{\bar \cS}\big) \vert = \phi_{\bar \cS} (\cA^f_n) + \phi_{\bar \cS} (\cB^f_n)$.
 \end{proof}

\begin{Defi}  Let $G$ be a finite group.\\
 We denote $B_n(G)$ the {\em Grothendieck group}  of $\cC^G_n$ with respect to the relations $\Gamma$ that is, the quotient   $$B_n(G):= \Omega_n(G)/\tilde \Omega_n(G) $$ of the free abelian group $\Omega_n(G)$ on the set of isomorphism classes of $(n,G)$-simplices by the subgroup $\tilde \Omega_n(G)$ generated by the formal differences $[\cX^f_n]- [\cA^f_n]-[\cB^f_n]$.

\end{Defi}

\begin{Rema}\label{Burn}${}$
\begin{itemize}
\item The construction satisfies the following property: if $\phi: \cC^G_n \ra A $ is a map from  $\cC^G_n$  to an abelian group $A$ given that $\phi (\cX^f_n )$ depends only on the isomorphism  class of $\cC^G_n$ and $\phi (\cX^f_n )= \phi (\cA^f_n )+\phi (\cB^f_n )$ for any element $(\cX^f_n , \cA^f_n ,\cB^f_n )$ in $\Gamma$, then there exists a unique $\bar \phi: B_n(G) \ra A$ such that $\phi (\cX^f_n)= \bar \phi ([\cX^f_n])$ for any $(n,G)$-simplex $\cX^f_n$.
 Let now  $(G/ \bar \cS)_n$ be a fixed $(n,G)$-simplex.\\ The function $\cX^f_n \mapsto   \big\vert  \Hom_{\cC_n^G} \big((G/ \bar \cS)_n,  \cX^f_n \big) \big\vert$ defined on the class of $(n,G)$-simplices (up to isomorphism), with values in $\ZZ$ extends to a group homomorphism $B_n(G) \ra \ZZ$.  
So, to see that $B_n(G)$ is non-trivial, it suffices to find $\cX_n^f$ with $\big \vert  \Hom_{} \big((G/ \bar \cS)_n,  \cX^f_n \big) \big \vert \neq 0$. Since 
$$\big \vert  \Hom_{} \big((G/ \bar \cS)_n,  (G/ \bar \cS)_n \big) \big\vert \neq 0 ,$$
 we have that $B_n (G) \neq 0$.
 \item In the special where $n=0$, one recovers the classical Burnside group $B(G)$ (see \cite{SBouc}), and for $n=1$, we have $B_1(G)= \Xi(G)$ the slice Burnside group introduced by Bouc in \cite{S.Bouc1}.
\end{itemize}
\end{Rema}
\begin{Prop}\label{opera}
The functor $d_j $, $j=0,\dots, n$   (resp.  $s_i $, $i=0,\dots, n-1$)
induces a group  homomorphism $$d_j : B_n(G) \ra B_{n-1}(G) ~~~~\big(\text{resp.}~   s_i: B_{n-1}(G) \ra B_{n}(G) \big) \footnote{Note that the maps $d_j $ and  $s_i $ should more appropriately be labeled, but we stick common practice and use $d_j $ and $s_i $ for face and degeneracy wherever we find them.}$$ such that the identities (\ref{1}), (\ref{2}), (\ref{3}) hold.
\end{Prop}
\begin{proof}
Indeed,  for any $(n,G)$-simplices $\cX^f_n$ with decomposition $[\cX^f_n, \cA^f_n, \cB^f_n]_k$,  i.e.
$\xymatrix{A_0 \sqcup B_0 \ar[r]^{f_1} &  A_{1} \sqcup B_1 \ar[r]^{f_{2}} &  \ar[r] \dots A_{k} \sqcup B_k\ar[r]^{f_{k+1}} & X_{k+1} \dots \ar[r]^{f_{n-1}}   &X_{n-1} \ar[r]^{f_{n}} &  X_{n}} $\\
then there is a decomposition of $d_j( \cX^f_n)$ as $[d_j( \cX^f_n),d_j(\cA^f_n), d_j( \cB^f_n )]_k$ if $k< j$ and\\  $[d_j( \cX^f_n), d_j(\cA^f_n), d_j( \cB^f_n )]_{k-1}$ if $k\ge j$. So the map which assigns an $(n,G)$-simplex $\cX^f_n$  to the image of  $d_j( \cX^f_n)$ in $B_{n-1}(G)$ induces, by the universal property, a well-defined map from $B_n(G)$ to $B_{n-1}(G)$.\\ 
  Similarly, there is a decomposition of $s_j( \cX^f_n)$ as $[s_j( \cX^f_n), s_j(\cA^f_n) , s_j( \cB^f_n )]_k$ if $k< j$ and  $[s_j( \cX^f_n), s_j(\cA^f_n), s_j( \cB^f_n )]_{k+1}$ if $k\ge j$.
  So  $s_j$  maps  the defining relations   $B_{n-1}(G)$ to those of  $B_n(G)$) and therefore gives a well-defined map (which is a group homomorphism) from $B_{n-1}(G)$ to $B_{n}(G)$. The equalities (\ref{1}), (\ref{2}), (\ref{3}) may be verified easily.

\end{proof}

\begin{Nota} ${}$
\begin{itemize}
\item Let $\pi(\cX^f_n)$ denote the image in $B_n(G)$ of the isomorphism class of $\cX^f_n$.
\item We set $\langle \bar \cG \rangle_G:= \pi((G/\bar \cG)_n)$, for any $n$-slice $\bar \cG=(G_0, G_{1}, \dots, G_n)$.
\item For any $(n,G)$-simplex $\cX_n^f$, we denote by $f(\cX^f_n)$ the $(n,G)$-simplex\\
$\xymatrix{X_0 \ar[r]^{f_1} & f_1( X_{0})\ar[r]^-{f_{2}}&\dots \ar[r]^-{f_{n}} &   f_{n}(f_{n-1}(f_{n-2}(...f_1( X_{0})...)) }.$

\end{itemize}

\end{Nota}
Then,

\begin{Lemm}\label{decomp}
Let $\cX_n^f$ be a $(n,G)$-simplex. Then in the group $B_n(G)$, we have 
$$\pi(\cX^f_n)= \sum\limits_{\substack{x \in [G\backslash X_0]}}  \langle \bar \cG_f^x \rangle_G,$$
where $\bar \cG_f^x$ denotes the $n$-slice $(G_x, G_{f_1(x)},\dots ,G_{f_{n}...f_{2}f_1(x)}  )$ and $G_{\bullet}$ denotes the stabilizer of the element $\bullet$.
\end{Lemm}

\begin{proof}
Note first  that in the group $B_n(G)$ we have that $\pi (\cX^f_n)= \pi(f(\cX^f_n))$. Indeed
writing $X_n= X_n \sqcup \emptyset$, we get  by the defining relations of $B_n(G)$ that
\begin{eqnarray*}
\pi (\cX^f_n)&=& \pi(\emptyset) +  \pi(\xymatrix{f_{1}^{-1}(X_1) \ar[r]^{f_1} & f_{2}^{-1}( X_{2})\ar[r]^-{f_{2}}&\dots  f_{n}^{-1}( X_{n}) \ar[r]^-{f_{n}} &   X_{n}})\\
&=& \pi(\xymatrix{f_{1}^{-1}(X_1) \ar[r]^{f_1} & f_{2}^{-1}( X_{2})\ar[r]^-{f_{2}}&\dots  f_{n}^{-1}( X_{n}) \ar[r]^-{f_{n}} &   X_{n}})\\
&=& \pi(f(\cX^f_n)).
 \end{eqnarray*}
Further, 
\begin{eqnarray*}
\pi(f(\cX^f_n)) &=&  \pi(\xymatrix{\coprod\limits_{\substack{x \in [G\backslash X_0]}} \cO_x \ar[r]^{f_1} & \coprod\limits_{\substack{x \in [G\backslash X_0]}} \cO_{f_1(x)}\ar[r]^-{f_{2}}&\dots \coprod\limits_{\substack{x \in [G\backslash X_0]}} \cO_{f_{n}(f_{n-1}(...f_1( x)...)) }})\\
&=& \sum\limits_{\substack{x \in [G\backslash X_0]}}  \pi( \emph{\textunderbrace{ \xymatrix{ \cO_x \ar[r]^{f_1} &  \cO_{f_1(x)}\ar[r]^-{f_{2}}&\dots \cO_{f_{n}(f_{n-1}(...f_1( x)...)) }} }{$\cX^{f\vert \cO}_n $}} )\\
\end{eqnarray*}
where the last equality follows from the defining relations of $B_n(G)$. Now, the $(n,G)$-simplex $\cX^{f\vert \cO}_n$ is obviously isomorphic to the indecomposable $(n,G)$-simplex $$\!\xymatrix{G/G_x\ar[r]^-{p_1} &  G/G_{f_1(x)}\ar[r]^-{p_{2}} &\dots   G/G_{f_{n}(f_{n-1}(...f_1( x)...)) }\!}.$$

\end{proof}

Since two indecomposable $(n,G)$-simplices $ (G/ \bar \cS)_n$ and  $(G/ \bar \cT)_n$ are isomorphic if and only if the $n$-slices $ \bar \cS$ and $\bar \cT$ are conjugate we have the following  corollary:

\begin{Coro}
The group $B_n(G)$ is generated by the elements $\langle \bar \cS \rangle_G$ where $\bar \cS $ runs through a set $[\Pi_n(G)]$ of representatives of conjugacy classes of $n$-slices of $G$. 
\end{Coro}

\begin{Prop}${}$
The product of $(n,G)$-simplices induces a commutative { ring structure} on $B_n(G)$ with identity  $\textbf e_n:= [\!\xymatrix{ \bullet \ar[r]^-{} &\dots \bullet \ar[r]^-{} &  \bullet \!}]$, where $\bullet$ is a $G$-set of cardinality $1$. This ring is called the {\em $n$-simplicial ring} of the finite group $G$.\\
Moreover,  the morphisms $d_j$ ($j=1, \dots, n$) and $s_i$ ($i=0, \dots, n-1$) in Proposition~\ref{opera} define morphisms of rings.
\end{Prop}

\begin{proof} 

 We have to show that the product of $(n,G)$-simplices induces a well-defined bilinear product $B_n(G) \times B_n(G) \ra B_n(G)$.
Let $\cX_n^f$ be a  $(n,G)$-simplex such that there is a decomposition $[\cX^f_n, \cA^f_n, \cB^f_n]_i$, and let $\cY^g_n$ be any $(n,G)$-simplex. We set $\cZ_n^h:= \cX_n^f \times \cY^g_n$. Then  the $(n,G)$-simplex $\cZ_n^h$ can be visualized  as follows
$$\!\xymatrix{(\!A_0 \times Y_0) \sqcup (B_0 \times Y_0) \ar[r]^-{f_0 \times g_0} & \dots (A_i \times Y_i) \!\sqcup \! (B_i \! \times \!Y_i) \!\ar[r]^-{f_{i} \! \times \! g_i} & X_{i+1}  \!  \times  \! Y_{i+1}\dots  \ar[r]^-{h_{n-1}} & \!   X_n \!  \times  \! Y_n\!},$$
since $(A_i \sqcup B_i) \times Y_i = \underbrace{(A_i \times Y_i)}_{C_i} \sqcup \underbrace{ (B_i \! \times \!Y_i)}_{D_i}  $.
Hence 
\begin{eqnarray*}
\pi(\cZ_n^h)&=& \pi(\!\xymatrix{C_0 \sqcup D_0\ar[r]^-{h_0} & \dots C_i \sqcup D_i \!\ar[r]^-{h_{i}} & X_{i+1}  \!  \times  \! Y_{i+1}\dots  \ar[r]^-{h_{n-1}} & \!   X_n \!  \times  \! Y_n\!}))\\
&=&  \pi(\!\xymatrix{C_0\ar[r]^-{h_0} & \dots C_i \!\ar[r]^-{h_{i}} & X_{i+1}  \!  \times  \! Y_{i+1}\dots  \ar[r]^-{h_{n-1}} & \!   X_n \!  \times  \! Y_n\!}))\\
&+&  \pi(\!\xymatrix{D_0\ar[r]^-{h_0} & \dots D_i \!\ar[r]^-{h_{i}} & X_{i+1}  \!  \times  \! Y_{i+1}\dots  \ar[r]^-{h_{n-1}} & \!   X_n \!  \times  \! Y_n\!}))\\
&=&  \pi(\!\xymatrix{ f_1^{-1}(A_1) \times g_1^{-1}(Y_1) \ar[r]^-{h_0}&f_2^{-1}(A_2) \times g_2^{-1}(Y_2) \dots \!\ar[r]^-{h_{i}} & A_{i}  \!  \times  \! Y_{i}\dots  \ar[r]^-{h_{n-1}} & \!   X_n \!  \times  \! Y_n\!}))\\
&+&  \pi(\!\xymatrix{ f_1^{-1}(B_1) \times g_1^{-1}(Y_1) \ar[r]^-{h_0}&f_2^{-1}(B_2) \times g_2^{-1}(Y_2) \dots \!\ar[r]^-{h_{i}} & A_{i}  \!  \times  \! Y_{i}\dots  \ar[r]^-{h_{n-1}} & \!   X_n \!  \times  \! Y_n\!}))\\
\end{eqnarray*}
Clearly, we have (it suffices to set $Y_i=Y_i \sqcup \emptyset$)
$$  \pi(\!\xymatrix{ g_1^{-1}(Y_1) \ar[r]^-{g_0}& g_2^{-1}(Y_2) \dots \!\ar[r]^-{g_{i}} &   \! Y_{i}\dots  \ar[r]^-{g_{n-1}} & \!   \! Y_n\!})= \pi(\cY^g_n).$$
So $\pi(\cZ_n^h)= \pi(\cA_n^f \times \cY^g_n ) + \pi(\cB_n^f \times \cY^g_n )$,
this shows that the product preserves the defining relations of $B_n(G)$. Now it is obvious to see that the product is associative, commutative and admits $\textbf e_n$ as an identity element.\\
 Let $\cY^g_n$ be $(n,G)$-simplex. Since $(f_{j+1}\circ f_j) \times (g_{j+1}\circ g_j)=(f_{j+1} \times g_{j+1})\circ (f_{j} \times g_{j})$, we have that 
$$d_j\big( \cX^f_n  \times \cY^g_n ) = d_j(\cX^f_n  ) \times d_j( \cY^g_n ) , ~~ s_j\big( \cX^f_n  \times \cY^g_n ) = s_j(\cX^f_n  ) \times s_j( \cY^g_n ) $$
and so  the induced maps $d_j$ and $s_j$ are ring homomorphisms.
  Furthermore $d_j (\textbf e_n)= \textbf e_{n-1}$ and $s_j (\textbf e_n)= \textbf e_{n+1}$.   
\end{proof}

\begin{Rema}

Letting $\Delta$  the simplicial category that is $\Delta$ has as objects all finite ordinal numbers $[n]=\{0,1, \dots, n\}$
and as morphisms $f: [n] \ra [m]$ all monotone maps; that is, the maps $f$ such that  $f(i)\le f(j)$ if $i<j$.\\
 Then an induced functor $\cC_n^G \ra \cC_m^G$  for any finite $G$ and an induced $G$-equivariant map $\Pi_n(G) \ra \Pi_{m}(G) $  and a group homomorphism $B_n(G) \ra B_{m}(G)$ given by, $$\cX^f_n  \mapsto  s^{n-1}_{j_k} \dots s^{n-k}_{j_1}d^{n-k}_{i_1}\dots d^m_{i_h} \cX^f_n$$ where $k$ and $h$ satisfy the following conditions: 
\begin{itemize}
\item $n+k-h=m$,
\item $0 \le i_1 < \dots <i_k \le m$,
\item $0 \le j_1 < \dots j_h \le n$,
\item $i_1, \dots,i_k $ are the elements of $[m]$ not in the image of $f$,
\item $j_1, \dots ,i_h $ are the elements of $[n]$ at which $f$ does not increase.
\end{itemize}
\end{Rema}



\begin{Prop}\label{produit} Using the generators of $B_n(G)$, the multiplication is given by

$$\langle \bar \cS \rangle_G . \langle \bar \cT \rangle_G= \sum\limits_{\substack{g \in [S_0\backslash G/ T_0]}}  \langle \bar \cS \cap {}^{g}\bar \cT   \rangle_G$$
\end{Prop}
\begin{proof}
Note that each $G$-orbit of the $G$-set $(G/\cS_0) \times (G/ \cT_0)$ determines a double coset $\cS_0g\cT_0$, in the following way:
$$\cO_{(xS_0, yT_0)} \ra S_0x^{-1}yT_0.$$  Conversely, 
the $G$-orbit of $(G/\cS_0) \times (G/\cT_0)$ corresponding to $\cS_0a\cT_0$ consists of all distinct pairs in the collection $\{(x \cS_0, y\cT_0)~\vert ~x^{-1}y \in \cS_0a \cT_0\}$. The stabilizer of the pair  $(\cS_0, g \cT_0)$ is precisely $\cS_0\cap {}^g \cT_0$. Therefore, the orbit of $(G/\cS_0) \times (G/\cT_0)$ corresponding to $\cS_0g \cT_0$ is isomorphic (as $G$-set) to $G/(\cS_0\cap {}^g \cT_0)$, and therefore 
$$(G/\cS_0) \times (G/\cT_0) \cong \coprod\limits_{\substack{g \in  [\cS_0\backslash G/ \cT_0]}} G/(\cS_0 \cap {}^g\cT_0)$$
On the other hand, the image of $(\cS_i, g\cT_i)$ by the map $$(G/\cS_i) \! \times \! (G/\cT_i)  \ra  (G/\cS_{i+1})\! \times \! \!(G/\cT_{i+1})$$ is the pair $(\cS_{i+1}, g\cT_{i+1})$, whose stabilizer is $\cS_{i+1} \cap {}^g \cT_{i+1}$. Hence inductively the result follows from Lemma~\ref{decomp}.
\end{proof}

\section{Ghost maps}
In this section, we examine a map $\Phi^G_n$ that this intrinsically related with the $n$-simplicial ring  $B_n(G$ in the sense that $\Phi^G_n$ may be discovered from the ring structure of $B_n(G)$.
\begin{Prop} \label{TheoBurn}${}$
\begin{enumerate}
\item For a $n$-slice $\bar \cS$ fixed,
the correspondence $\cX_n^f \mapsto \phi_{\bar \cS} \big(\cX^f_n\big)$  extends to a ring homomorphism still denoted by
$$  \phi_{\bar \cS}: B_n(G) \longrightarrow  \ZZ.$$
\item (An analogue of Burnside's Theorem) We let 
$$\Phi^G_n= ( \phi_{\bar \cS}): B_n(G) \longrightarrow  \prod_{\bar \cS \in [\Pi_n(G)]} \ZZ=C_n(G)$$
be the product of the $\phi_{\bar \cS}$.
Then $\Phi^G_n$ is an injective ring homomorphism, with finite cokernel as morphism of abelian groups.  The set $$\{\langle \bar \cS \rangle_G  ~\vert ~  \bar \cS  \in [\Pi_n(G)]\}$$ form a {\em basis} of $B_n(G)$.

\item  Let $R$ be an integral domain, $\phi: B_n(G) \ra R$ any ring homomorphism, and $$T(\phi) := \{ \bar \cS \in \Pi_n(G)~\vert ~ \phi(\langle \bar \cS \rangle_G) \neq 0  \}.$$ Then, there exists exactly one element $\bar \cK \in [\Pi_n(G)]$ that is minimal with respect to $\preceq$ in $T(\phi)$. Moreover, one has $\phi(x)= \phi_{\bar \cK}(x)\cdot 1_R$ for all $x \in B_n(G)$ and this minimal $\bar \cK$ in  $T(\phi)$.
\end{enumerate}
\end{Prop}

\begin{proof}
\begin{enumerate}
\item Corollary~\ref{defining} shows that $\phi_{\bar \cS}$ the defining relations of $B_n(G)$ are mapped to $0$ by $\phi_{\bar \cS}$. So it induces a well defined map from $B_n(G)$ to $\ZZ$. Now, since the product of $(n,G)$-simplices is the product of the category of $(n,G)$-simplices, it follows that 
$$\phi_{\bar \cS}\big( \cX_n^f . \cY_n^g\big) = \phi_{\bar \cS}(\cX_n^f ) \phi_{\bar \cS}(\cY_n^g).$$ The image of the identity $\textbf e_n$ by $\phi_{\bar \cS}$ is obviously $1$.
\item By definition $\Phi^G_n$ is a ring homomorphism. Suppose that $u \neq 0$ is in the kernel of $\Phi^G_n$. We write $u$ in terms of the generators $$u=  \sum\limits_{\substack{ \bar \cS \in [\Pi_n(G)]}} a_{\bar \cS} \langle \bar \cS \rangle_G.$$
We have a partial ordering on the $ \langle S_0 \dots S_n \rangle_G$ induced by $\preceq$.
Let $ \langle \bar \cS \rangle_G$ be maximal among the generators with $a_{\bar \cS}  \neq 0$. Then $\phi_{\bar \cS}( \langle \bar \cT \rangle_G ) \neq 0$  implies that $ \langle \bar \cS \rangle_G \preceq  \langle \bar \cT \rangle_G$. Hence $$0= \phi_{\bar \cS} (u)=a_{\bar \cS} \phi_{\bar \cS}( \langle \bar \cS \rangle_G )\neq 0,$$
a contradiction.\\
Now by the first part of the proof, $B_n(G)$ has $\ZZ$-rank $\vert [\Pi_n(G)] \vert =rank_{\ZZ} C_n(G)$, and since $\Phi^G_n$ is injective, then $\Im(\Phi^G_n)$ and $C_n(G)$ have the same rank, and so,  
the cokernel of $\Phi^G_n$ is finite.
\item The set $T(\phi)$ is not empty because $\textbf e_n \notin \Ker \phi$.\\ Let $\bar \cS$ be minimal in $T(\phi)$ with respect to the relation $\preceq$. Then by Proposition~\ref{produit},  for any $n$-slice $\bar \cT$
\begin{eqnarray*}
\langle \bar \cS \rangle_G.\langle \bar \cT \rangle_G&=&\sum\limits_{\substack{g \in [S_0\backslash G/ T_0]}}  \langle S_0 \cap {}^g T_0,\dots ,  S_n \cap {}^g T_n \rangle_G\\
\end{eqnarray*}
Since $\phi(\langle \bar \cS \rangle_G.\langle \bar \cT \rangle_G)= \phi(\langle \bar \cS \rangle_G).\phi(\langle \bar \cT \rangle_G) \neq 0$ in $R$, there exists $ g \in S_0\backslash G/ T_0$ such that $\phi(\bar \cS \cap {}^g \bar \cT) \neq 0$. Hence, by minimality of $\bar \cS$, we have that those $g$ ranges over the set $\{g \in  G/ T_0 ~\vert ~ S_n \le {}^g T_n  \dots \dots S_0 \le {}^g T_0 \}$. So $\phi(\langle \bar \cS \rangle_G).\phi(\langle \bar \cT \rangle_G )= \phi_{\bar \cS}\big(\langle \bar \cT \rangle_G \big)\phi(\langle \bar \cS \rangle_G$. Since $R$ is an integral domain, we can divide both sides by $\phi(\langle \bar \cS \rangle_G )$ and we then have $\phi(\langle \bar \cT \rangle_G)= \phi_{\bar \cS}\big(\langle \bar \cT \rangle_G \big)1_R$.
 So by linearity, $\phi(x)= \phi_{\bar \cS}(x)\cdot 1_R$ for all $x \in B_n(G)$. Now, if $\bar \cK$ is another minimal element, then $\phi(\langle \bar \cK \rangle_G)= \phi_{\bar \cS}(\langle \bar \cK \rangle_G)$ and $\phi( \langle \bar \cK \rangle_G)= \phi_{\bar \cK}(\langle \bar \cK \rangle_G)$. So $\bar \cS \preceq_G \bar \cK$ and by symmetry $\bar \cS =_G \bar \cK$.
 \end{enumerate}
\end{proof}
\begin{Rema} ${}$
\begin{itemize}
\item If $R$ is an integral domain, then two homomorphisms $\phi, \phi': B_n(G) \ra R$ coincide if and only if they have the same kernel. Moreover, any homomorphism $\phi: B_n(G) \ra R$ factors through some $\phi_{\bar \cS}: B_n(G) \ra \ZZ$ and the unique homomorphism $\ZZ \ra R$ given by $n \mapsto n.1_R$.
\item The ring $B_n(G)$ is finitely generated as  $\ZZ$-module, hence is a noetherian ring.\\ For each $n$-slice $\bar \cS$, the kernel of $\phi_{\bar \cS}$ is a prime ideal, since $\ZZ$ is an integral domain, and the intersection of all those kernels for $n$-slices $\bar \cS$ of $G$ is zero. In particular, the ring $B_n(G)$ is reduced (the intersection of all the prime ideals of $B_n(G)$  is zero).
\item Since $C_n(G)$ is  an integral over $B_n(G)$, the Theorem of Cohen-Seidenberg implies that their Krull dimensions coincide, and every prime ideal of $B_n(G)$ comes from $C_n(G)$.

\item The point 2) of the proposition shows in particular that the set
$$\{ \langle \bar \cS \rangle_G ~ \vert ~  \bar \cS \in [\Pi_n(G)] \}$$ form a basis of $B_n(G)$
and a  $(1,G)$-simplicial Burnside  can also be seen as the lattice Burnside ring of some lattice introduced by Oda, Takegahara and Yoshida in \cite{OTY}.

\end{itemize}
\end{Rema}

We have the following corollary,
\begin{Coro}
The map $$ \QQ \Phi^G_n: \QQ B_n(G) \ra \prod_{\bar \cS \in [\Pi_n(G)]} \QQ,$$
where $\QQ \Phi^G_n$ is the rational extension of $\Phi^G_n$, is an isomorphism of $ \QQ$-algebras.
\end{Coro}
Note that every $\QQ$-algebra homomorphism $\QQ B_n (G) \ra \QQ$ is of the form $\QQ \phi_{\bar \cS}$ for some $\bar \cS \in \Pi_n(G)$. For $\bar \cS , \bar \cT  \in \Pi_n(G)$, we have  $\QQ \phi_{\bar \cS}= \QQ \phi_{\bar \cT}$ if and only $\bar \cS=_G \bar \cT$.\\
More generally,
\begin{Nota} Put $W_G(\bar \cS)= N_G(\bar \cS)/S_0$.
Let $p$ be a prime, and $\infty$ be just a symbol. For each $\ZZ$-module $M$, we shall set $M_{(p)} = \ZZ_{(p)} \otimes_\ZZ M$, where $ \ZZ_{(p)}$ is the localisation of $\ZZ$ at $p$, and $M_{(\infty)}=M$. For a $n$-slice $\bar \cS$, we denote $W_G(\bar \cS)_p$ a Sylow $p$-subgroup of $W_G(\bar \cS)$, and set $W_G(\bar \cS)_{\infty} =W_G(\bar \cS)$. Let $ \big( \Phi_{n}^G \big)_p$ or simply $\Phi_{(p)}^G$ (if there is no risk of confusion)   denote the  homomorphism of $ \ZZ_{(p)}$-modules from ${B_n(G)}_{(p)}$  to ${C_n(G)}_{(p)}$ induced by $\Phi_n^G$. 
\end{Nota}

\begin{Prop}
  Then 
\begin{enumerate}
\item We consider     ${B_n(G)}_{(p)}$ and ${C_n(G)}_{(p)}$ as  subrings of   $\prod_{\bar \cS \in [\Pi_n(G)]} \QQ$.

The set $J' := \big\{ \frac{1}{\phi_{\bar \cS} (\bar \cS)}\bar \cS:= \big(\frac{\phi_{\bar \cT} (\bar \cS)}{\phi_{\bar \cS} (\bar \cS)} \big)_{\bar \cT \in [\Pi_n(G)]} ~\vert~  \bar \cS \in [\Pi_n(G)] \big\}$ is a basis of $C_n(G)$.
\item If we define $Obs(G)$ as $ \underset{\bar \cS \in [\Pi_n(G)]} {\bigoplus} \ZZ \big/\vert W_G(\bar \cS)  \vert\ZZ$, then $$C_n(G)_{(p)}/ \Im \Phi^G_{(p)} \cong Obs(G)_{(p)}$$
\item Define  a homomorphism of $ \ZZ_{(p)}$-modules           $\Psi_{(p)}^G: C_n(G)_{(p)} \ra   Obs(G)_{(p)}$ by 

$$ (x_{\bar \cS})_{\bar \cS \in [\Pi_n(G)]} \mapsto \big( \sum_{gT_0 \in W_G(\bar \cT)_{(p)}}  x_{<g>\bar \cT} ~~\mod ~ \vert W_G(\bar \cT)_{(p)} \vert   \big)_{\bar \cT \in [\Pi_n(G)]}$$ Then $\Psi_{(p)}^G$ is surjective.
\item  The sequence 
$$\xymatrix{ 0 \ar[r]&  {B_n(G)}_{(p)}  \ar[r]^{ \Phi_{(p)}^G} &{C_n(G)}_{(p)} \ar[r]^-{\Psi_{(p)}^G} &  Obs(G)_{(p)} \ar[r]&0 }$$ of $ \ZZ_{(p)}$-modules is exact.

\end{enumerate}
\end{Prop}


\begin{proof}

\begin{enumerate}
\item  By Lemma~\ref{conj} ~$4)$, we have that $\phi_{\bar \cS} (\bar \cS)$ divides $\phi_{\bar \cT} (\bar \cS)$, and so, $\big(\frac{\phi_{\bar \cT} (\bar \cS)}{\phi_{\bar \cS} (\bar \cS)} \big)_{\bar \cT \in [\Pi_n(G)]}$ is an element of ${C_n(G)}_{(p)}$, for any $ \bar \cS \in [\Pi_n(G)]$. Now, compare the set $J'$ with the canonical basis
$$J:= \big\{ i_{\cS}= \big( \delta(\bar \cS ,\bar \cT  ) \big)_{\bar \cT \in [\Pi_n]} ~~\vert~~  \bar \cS \in [\Pi_n]  \big\}$$  of ${C_n(G)}_{(p)}$ where $\delta$ is the Kronecker's symbol, i.e.        
$ \delta(\bar \cS ,\bar \cT  )=  \left\{
\begin{array}{ccc}
  1& \mbox{for}  &  \bar \cS =\bar \cT \\
0&  \mbox{for}   & \bar \cS \neq \bar \cT 
\end{array}
\right.$\\
  Since $\vert J\vert=\vert J' \vert=\vert \Pi_n(G)\vert$, it suffices to prove that each $ i_{\bar \cS}$ can be written as an integral  combination of the  $\frac{1}{\phi_{\bar \cS} (\bar \cS)}\bar \cS$. This is done by induction with respect to $\preceq$.
  If $\bar \cS= (1, \dots, 1)$ then $\frac{\phi_{\bar \cT} (\bar \cS)}{\phi_{\bar \cS} (\bar \cS)}= \delta(\bar \cS ,\bar \cT  )$, for any $\bar \cT$, and so, $ i_{\bar \cS} \in J'$. For arbitrary $\bar \cS$, we have $\frac{\phi_{\bar \cS} (\bar \cS)}{\phi_{\bar \cS} (\bar \cS)}=1$ and $\frac{\phi_{\bar \cT} (\bar \cS)}{\phi_{\bar \cS} (\bar \cS)}=0$ for $\bar \cT \npreceq \bar \cS$, and so, $\frac{1}{\phi_{\bar \cS} (\bar \cS)}\bar \cS=  i_{\bar \cS} + \sum\limits_{\substack{\bar \cT \in [\Pi_n(G)]\\ \bar \cT \prec \bar \cS }} n_{\bar \cT,\bar \cS}  i_{\bar \cT}$ with $ n_{\bar \cT,\bar \cS} \in \ZZ$.\\
  Now by induction hypothesis, any $i_{\bar \cT} $ with $\bar \cT \prec \bar \cS$ is an integral linear combination of the elements of $J'$. So, the same is true for $ i_{\bar \cS}.$
  
  \item By Proposition~\ref{TheoBurn}~2), we have $\Im  \Phi_{(p)}^G = \bigoplus_{\bar \cS  \in [\Pi_n(G)] } \Phi_{(p)}^G(\langle \bar \cS \rangle_G ) $ and by Assertion 1) $${C_n(G)}_{(p)}=  \bigoplus_{\bar \cS  \in [\Pi_n(G)] } \frac{1}{\vert W_G(\bar \cS)_p \vert }  \Phi_{(p)}^G(\langle \bar \cS \rangle_G) \ZZ.$$ Hence $C_n(G)_{(p)}/ \Im \Phi^G_{(p)} \cong Obs(G)_{(p)}$.

\item Let $\bar i_{\bar \cS}= \big( \delta(\bar \cS ,\bar \cT  ) \mod \vert W_G(\bar \cS)_p \vert \big)_{\bar \cT \in [\Pi_n]}$. Obviously, the elements $\bar i_{\bar \cS}$  for $\bar \cS \in [\Pi_n]$ form a $\ZZ_{(p)}$-basis of $Obs(G)_{(p)}$. Now set 
$$R_0=\{  \bar \cS \in [\Pi_n] ~\vert~  \bar i_{\bar \cS} \notin \Im \Psi_{(p)}^G    \}$$
Suppose that $R_0 \neq \emptyset$, and let $\bar \cS$ be a minimal element of $R_0$ with respect to $\preceq_G$. Then no element $\bar \cT$ of $R_0- \{ \bar \cS\}$ satisfies $\bar \cT \preceq_G \bar \cS$, and so $\Psi_{(p)}^G  ((i_{\bar \cT})_{\bar \cT \in [\Pi_n] })=(y_{\bar \cT})_{\bar \cT \in [\Pi_n] } $, where 
 \[
y_{\bar \cT}=  \left\{
\begin{array}{ccc}
  1 ~\mod ~ \vert W_G(\bar \cS)_p \vert  & \mbox{for}  &  \bar \cT =\bar \cS \\
0~ \mod~  \vert W_G(\bar \cS)_p \vert &  \mbox{for}   &  \bar \cT  \in R_0- \{ \bar \cS\}
\end{array}
\right.\\
\]
But, $\bar i_{\bar \cT} \in \Im \Psi_{(p)}^G $ for any $\bar \cT \notin R_0$, which yields $\bar i_{\bar \cS} \in \Im \Psi^G_n$. This is a contradiction. Consequently, we have $R_0 =\emptyset$, and so $\Psi^G_{(p)}$ is surjective.
\item Let $\bar \cS \in [\Pi_n(G)]$. Then $$ \Psi_{(p)}^G( \Phi_{(p)}^G (\langle \bar \cS \rangle_G))= \big( \sum_{rT_0 \in W_G(\bar \cT)_{(p)}}  \vert \Inv_{<r>\bar \cT}(\bar \cS)\vert  ~~\mod ~ \vert W_G(\bar \cT)_{(p)} \vert   \big)_{\bar \cT \in [\Pi_n(G)]}$$
where $ \Inv_{<r>\bar \cT}(\bar \cS) = \{gS_0 \in G/S_0~\vert~ <r> \bar \cT \preceq {}^g\bar \cS \} := I_{<r>\bar \cT}$. Set $W= W_G(\bar \cS)_p $. Then $\Inv_{\bar \cT}(\bar \cS) $  can be view as a left $W$-set by the action given by $rT_0. gS_0=rgS_0$. With this action, one has
$ \Inv_{<r>\bar \cT}(\bar \cS) = \{gS_0 \in \Inv_{\bar \cT}(\bar \cS)~\vert~ rT_0 \cdot gS_0=gS_0 \}$. So
\begin{eqnarray*}
\sum_{rT_0 \in W}  \vert \Inv_{<r>\bar \cT}(\bar \cS)\vert&=& \sum_{gS_0 \in I_{\bar \cT}} \{rT_0 \in W~\vert~  rT_0 \cdot gS_0=gS_0 \}\\
&=& \sum_{gS_0 \in I_{\bar \cT}} \vert W_{gS_0}\vert \\
&=&  \sum_{gS_0 \in [W \backslash I_{\bar \cT}]} \vert O_{gK_0}\vert \vert W_{gS_0}\vert\\
&~~\equiv& 0 ~\mod ~\vert W \vert,
\end{eqnarray*}
and so $\Im  \Phi_{(p)}^G\subseteq \Ker \Psi_{(p)}^G$. It remains to prove that $  \Ker \Psi_{(p)}^G \subseteq \Im  \Phi_{(p)}^G$. Now, since $\Psi_{(p)}^G$ is surjective and $\Psi_{(p)}^G \circ \Phi_{(p)}^G =0$, we have that $\Psi_{(p)}^G$ factorizes through the cokernel of $\Phi_{(p)}^G$, which is isomorphic by $Obs_{(p)}$ by $2)$. So, we obtain a surjective map $\Coker \Phi_{(p)}^G \ra Obs_{(p)}$ between two isomorphic groups. This map is then an isomorphism, and so $\Ker \Psi_{(p)}^G $ is equal to $\Im  \Phi_{(p)}^G$.

\end{enumerate}

\end{proof}

\section{Idempotents elements}
By the isomorphism $\QQ \Phi$, there is an element $e^G_{\bar \cT} \in \QQ B_n(G)$ for each $n$-slice  $\bar \cT$ of $G$ such that  

$$\QQ\phi_{\bar \cS}^G (e^G_{\bar \cT} )= \left\{ \begin{array}{rcl}
1& \mbox{if}
& \bar \cS=_G \bar \cT \\ 0 & \mbox{otherwise} & 
\end{array}\right.$$
Clearly, the set $\{e^G_{\bar \cT}  ~ \vert ~  \bar \cT \in [\Pi_n(G)] \}$ is the set of primitive idempotents of $\QQ B_n(G)$.\par
In order to give the explicit formula of the primitive idempotent $e^G_{\bar \cT}$, we need  the Möbius function of  a  finite poset $(\cX, \le)$.
The Möbius function $\mu_\cX: \cX \times \cX \ra \ZZ$ of a finite poset is defined  inductively as follows (see \cite{Stanley}):
$$\mu_\cX(x,x)=1, ~~~ \mu_\cX(x,y)=0 ~\text{if}~ x \nleq y, ~~ \sum\limits_{\substack{t \le y}} \mu_\cX(x,t)=0  ~\text{if}~ x < y.$$ 

\begin{Prop}
Let $ \bar \cS$ be a $n$-slice of $G$. Then  the explicit formula of the primitive idempotent $e^G_{ \bar \cS}$ is given by
$$e^G_{ \bar \cS}= \frac{1}{\vert N_G(\bar \cS)\vert} \sum_{\bar \cT \preceq \bar \cS} \vert {\cT}_0 \vert \mu_{\Pi} \big( \bar \cT, \bar \cS \big)\langle \bar \cT \rangle_G,$$
where  $\mu_{\Pi}$ is the Möbius function of the poset $(\Pi_n(G), \preceq)$.
\end{Prop}
\begin{proof}

Let $\bar \cK$ be any $n$-slice of $G$. Then 
\begin{eqnarray*}
\QQ\phi^G_{\bar \cK}( \vert G \vert e^G_{\bar \cS}) &=& \frac{\vert G \vert}{\vert N_G(\bar \cS)\vert}\! \sum_{\bar \cT \preceq \bar \cS} \! \vert \bar \cT_0 \vert \mu_{\Pi} \big( \bar \cT, \bar \cS \big) \QQ\phi^G_{\bar \cK} (\langle \bar \cT \rangle_G)\\
&=&  \frac{\vert G \vert}{\vert N_G(\bar \cS)\vert}\! \sum_{\bar \cT \preceq \bar \cS} \! \vert \bar \cT_0 \vert \mu_{\Pi} \big( \bar \cT, \bar \cS \big) \vert \{ g \in G/\bar \cT_0~\vert ~ \bar \cK \preceq {}^g\bar \cT\} \vert\\
&=&\frac{\vert G \vert}{\vert N_G(\bar \cS)\vert}\! \sum_{g\in G} \sum\limits_{\substack{\bar \cK \preceq {}^g\bar \cT \preceq {}^g\bar \cS}}\! \mu_{\Pi} \big( \bar \cT, \bar \cS \big)\\
&=&\vert G \vert \delta(\bar \cS,  \bar \cT)
\end{eqnarray*}
where the second equality follows from Lemma~\ref{conj}.
By the property of the Möbius function, we have that the sum $\sum\limits_{\substack{\bar \cK \preceq {}^g\bar \cT \preceq {}^g \bar \cS}}\! \mu_{\Pi} \big( \bar \cT, \bar \cS \big)$ is zero unless $\bar \cK={}^g\bar \cS$, for any $g \in G$. Therefore,
 $$\phi_{\bar \cK}^G (e^G_{\bar \cS})= \left\{ \begin{array}{rcl}
1& \mbox{if}
&\bar \cK=_G \bar \cS \\ 0 & \mbox{otherwise} & 
\end{array}\right.$$
\end{proof}

In the case $n=0$ the formula of the primitive idempotent was given by Gluck in \cite{Gluck} and independently by Yoshida in \cite{Yoshida}. The formula in the case $n=1$ was stated by Bouc in \cite{S.Bouc1}. The idempotent $e^G_{\bar \cS}$ is the only idempotent of $\QQ B_n(G)$  with the following property:  

\begin{Prop} [Characterization of $e^G_{\bar \cS}$]
\item  Let $\bar \cS$ be  a $n$-slice of $G$. Then  $\mathcal X . e^G_{\bar \cS}=  \QQ\phi_{\bar \cS}^G(\mathcal X)  e^G_{\bar \cS},$ for any $\cX \in \QQ B_n(G)$.\\
Conversely, if $\mathcal Y \in \QQ B_n(G)$ is such that $\mathcal X \cdot \mathcal Y=  \QQ\phi_{\bar \cS}^G(\mathcal X) \mathcal Y,$
then $\mathcal Y \in \QQ e^G_{\bar \cS}$ (that is $\mathcal Y $ is a rational multiple of $ e^G_{\bar \cS}$).

\end{Prop}

\begin{proof}
\item Note that a $\QQ$-basis of $\QQ B_n(G)$ is given by the $e^G_{\bar \cT}$ where $\bar \cT$ runs through the set $\Pi_n(G)$ of conjugacy classes of $n$-slices of $G$. So for any $\mathcal X \in \QQ B_n(G)$, we have that 
$$\cX= \sum_{\bar \cT \in [\Pi_n(G)]} \lambda_{\bar \cT} e^G_{\bar \cT},$$
where $\lambda_{\bar \cT}$ are rational numbers.
Since the  elements $e^G_{\bar \cT}$ are orthogonal, it follows that for any $n$-slice  $\bar \cS \in [\Pi_n(G)]$ we have,
$$\phi_{\bar \cS}^G(\mathcal X)= \lambda_{\bar \cS}~~ \text{and} ~~\cX.e^G_{\bar \cS}= \phi_{\bar \cS}^G(\mathcal X)e^G_{\bar \cS}.$$
Conversely, let $\cY$ be an element of $\QQ B_n(G)$ verifying $\cX . \cY = \phi_{\bar \cS}^G(\mathcal X) \mathcal Y$ for any $\cX \in  \QQ B_n(G)$.
Then  in particular $e^G_{\bar \cT}. \cY=0$ if $\bar \cS \neq_G \bar \cT$, thus $\cY= \phi^G_{\bar \cS} (\cY) e^G_{\bar \cS}$  is a rational multiple of $e^G_{\bar \cS}$.


\end{proof}

\begin{Prop} Let $(\cX, \le )$ be a finite poset. Let $\Pi_n(\cX)$ denote the set of $n$-tuples $(x_0, \dots, x_n)$ of elements of $\cX$ such that $x_0 \le \dots \le x_n$. Define a partial order $\preceq$ on $\Pi_n(\cX)$ by 
$$ (x_0, \dots, x_n) \preceq (y_0, \dots, y_n) \Leftrightarrow x_i \le y_i, \forall i=0, \dots, n$$
Then the Möbius function $\mu_\Pi$ of the poset $(\Pi_n(\cX), \preceq)$ can be computed as follows, for any $\bar x:=(x_0, \dots, x_n) , \bar y:=(y_0, \dots, y_n)  \in \Pi_n(\cX)$:

\begin{equation}\label{cinq}
\mu_{\Pi} \big( \bar x , \bar y \big)= \left\{ \begin{array}{rcl}
\prod\limits_{\substack{ i=0}}^n\mu_{\cX}(x_i,y_i)& \mbox{if}
&x_0 \le y_0 \le x_1 \le y_1 \dots \le x_n \le y_n \\ 0 & \mbox{otherwise} & 
\end{array}\right.
\end{equation}
where $\mu_\cX$ is the Möbius function of the poset $(\cX, \le)$.
\end{Prop}
\begin{proof}
Let $m\big( \bar x , \bar y \big)$ denote the expression  defined by the right hand side of (\ref{cinq}). Then if $\bar x \preceq \bar y$ 
\begin{eqnarray*}
\sum\limits_{\substack{ \bar t \in \Pi_n(\cX) \\  \bar x \preceq \bar t \preceq \bar y}} m \big(\bar x,  \bar t  \big)&=&  \sum\limits_{\substack{ (t_0, \dots, t_n) \in \cP_n}} \mu_\cX (x_0, t_0) \dots \mu_\cX (x_n, t_n)\\
\end{eqnarray*}
where $\cP_n := \{ \bar t \in \Pi_n(\cX) ~\vert~ \left\{ \begin{array}{rcl}
x_i \le t_i \le y_i& \mbox{for all $i=0, \dots, n$}\\
 x_0 \le \dots \le x_n & \\
   y_0 \le \dots \le y_n & \\
 t_i \le x_{i+1} & \mbox{for all $i=0, \dots, n-1$}
\end{array}\right.\}$.\\
So,
$$\sum\limits_{\substack{ (t_0, \dots, t_n) \in \cP_n}} \mu_\cX (x_0, t_0) \dots \mu_\cX (x_n, t_n) =$$$$ \big( \sum_{x_n \le t_n \le y_n} \mu_\cX (x_n, t_n)\big)\big( \sum_{(t_0, \dots, t_{n-1}) \in \cP'_{n-1}} \mu_\cX (x_0, t_0) \dots \mu_\cX (x_{n-1}, t_{n-1}) \big) ~~~~(\star)$$
where $\cP'_{n-1} := \{ (t_0, \dots, t_{n-1}) \in \Pi_{n-1}(\cX) ~\vert~ \left\{ \begin{array}{rcl}
x_i \le t_i \le y_i& \mbox{for all $i=0, \dots, {n-1}$}\\
 t_i \le x_{i+1} & \mbox{for all $i=0, \dots,  n-1$}
\end{array}\right.\}$.\\
The first factor of $(\star)$ is equal to $0$ if $x_n \neq y_n$, and $1$ if $x_n = y_n$. Hence, if $x_n = y_n$, then the second factor is equal to 
$$ \sum_{(t_{0}, \dots, t_{n-1}) \in \cP_{n-1}} \mu_\cX (x_0, t_0) \dots \mu_\cX (x_{n-1}, t_{n-1}) $$
Hence inductively, we have 
\begin{eqnarray*}
\sum\limits_{\substack{ \bar t \in \Pi_n(\cX) \\  \bar x \preceq \bar t \preceq \bar y}} m \big(\bar x,  \bar t \big)&=&  \prod\limits_{\substack{ i=0}}^n \delta(x_i, y_i)\\
\end{eqnarray*} 
and the proposition follows.
\end{proof}

\begin{Coro}
Let $ \bar \cS= (S_0, \dots, S_n)$ and $\bar \cT=(T_0, \dots, T_n)$ be $n$-slices of $G$. Then 

$$\mu_{\Pi} \big( \bar \cT, \bar \cS\big)= \left\{ \begin{array}{rcl}
\prod\limits_{\substack{ i=0}}^n\mu(T_i,S_i)& \mbox{if}
&T_0 \le S_0 \le T_1 \le S_1 \dots \le T_n \le S_n \\ 0 & \mbox{otherwise} & 
\end{array}\right.$$
where $\mu$ is the Möbius function of the poset of subgroups of $G$.\\
In particular  
$$e^G_{\bar \cS}= \frac{1}{\vert N_G(\bar \cS)\vert} \sum_{T_0 \le S_0 \le \dots \le T_n \le S_n} \vert T_0 \vert \mu  (T_0, S_0) \dots \mu( T_n, S_n)\langle \bar \cT \rangle_G,$$

\end{Coro}


\section{Prime ideals}
The aim of this part is the proof that a finite group $G$ is solvable if and only if the prime ideal spectrum of $B_n(G)$ is connected, i.e., if and only if $0$ and $1$  are the only idempotents in $B_n(G)$, like established by A.Dress in \cite{Dress} for $B_0(G)$ and S.Bouc in~\cite{S.Bouc1}.
Note that $C_n(G)$ is integral over $B_n(G)$,  because it  is generated by idempotent elements which are integral over any subring.   Hence by the going-up theorem, every prime ideal of $B_n(G)$ comes from  $C_n(G)$.\\
\begin{Prop}\label{ideals}${}$
Let $p$ denote either $0$ or a prime number.
If $ \bar \cS \in \Pi_n(G)$, let $I_{\bar \cS,p}$ be the prime ideal of $B_n(G)$ defined as the kernel of the ring homomorphism
$$\xymatrix{ B_n(G) \ar[r]^-{\phi^G_{\bar \cS}} & \ZZ \!\ar[r]^-{p} & \ZZ/p\ZZ}.$$
Then every prime ideal $I$ of $B_n(G)$  has the form $I_{\bar \cS,p}$ for a suitable $\bar \cS \in \Pi_n(G)$.
Moreover,  given a prime ideal $I$ there exists a unique $\bar \cK \in [\Pi_n(G)]$ with $I= I_{\cK,p}$ and  $\phi_{\bar \cK}^G(\langle \bar \cK \rangle_G) \neq 0 ~(\mod ~p) $ where $p$ is the characteristic of the ring $R:= B_n(G)/I$.
\end{Prop}

\begin{proof} Consider the natural map $\phi: B_n(G) \ra B_n(G)/I$. Then  $\phi(x)= \phi_{\bar \cS}(x)\cdot 1_R$ for some $n$-slice $\bar \cS$ by Proposition~\ref{TheoBurn}.
So
\begin{eqnarray*}
I&=&\{x \in B_n(G) ~\vert~ \phi(x)=0\}\\
&=&\{x \in B_n(G) ~\vert~ \phi_{\bar \cS}(x)\cdot 1_R=0\}\\
&=&\{x \in B_n(G) ~\vert~ \phi_{\bar \cS}(x) \equiv 0 (\mod ~p)\}= I_{ \bar \cS,p} 
\end{eqnarray*}
where $p$ is the characteristic of the ring $B_n(G)/I$.\\
Let $\cI_n:=\{\bar \cS \in \Pi_n(G)~\vert~ \langle \bar \cS \rangle_G \notin I\}$. If $\bar \cT$ and $\bar \cK$ are both minimal in $\cI$, then 
\begin{eqnarray*}
\langle \bar \cT \rangle_G.\langle \bar \cK \rangle_G&=& \sum_{g \in [T_0\backslash G/K_0]}\langle \bar \cT \cap {}^g\bar \cK   \rangle_G\\
&=& \phi_{\bar \cT}(\langle \bar \cK \rangle_G)\langle \bar \cT \rangle_G ~(\mod~I)
\end{eqnarray*}
(This relation must hold for any $\langle \bar \cK \rangle_G \notin I$)\\
Since $\langle \bar \cT \rangle_G.\langle \bar \cK \rangle_G \notin I$, we have $ \phi_{\bar \cT}(\langle \bar \cK \rangle_G) \neq 0$ and $ \phi_{\bar \cK}(\langle \bar \cT \rangle_G) \neq 0$ by symmetry, and so $\bar \cK=_G \bar \cT$. So the set $\cI_n$ has a unique minimal element, up to isomorphism.  On the other hand, the  quotient ring $B_n(G)/I$ is an integral domain, we have that either $p=0$ or  $p$ is a prime number. If $p=0$, then the projection $\phi: B_n(G) \ra B_n(G)/I$ is equal to $\phi_{\bar \cK}$ and   $I= I_{ \cK,p}$. If $p \neq 0$, then $\phi$ is equal to the reduction  of $\phi_{\bar \cK}$ modulo $p$.\\
If $\bar \cG$ is any $n$-slice of $G$ with $I=  I_{\cG,p}$ and $\phi_{\bar \cG} (\langle \bar \cG \rangle_G)  \neq 0~~(\mod ~p)$  then for a $\bar \cK$ minimal  in $\cI_n$

 $$\phi_{\bar \cK} (\langle \bar \cK \rangle_G) \equiv \phi_{ \bar \cG} (\langle \bar \cK \rangle_G) \neq 0~~(\mod~p).$$ In particular $\phi_{\bar \cK} (\langle \bar \cG \rangle_G) $ is non zero; and similarly $\phi_{\bar \cG} (\langle \bar \cK \rangle_G) $ is non zero. This can only occur if $\bar \cK =_G \bar \cG$.
\end{proof}

\begin{Nota}
Let $p$ be a prime number.
\begin{itemize}
\item Let $\big(\Pi_n(G) \big)_{(p)}$ denote the subset of $\Pi_n(G)$ consisting of the $n$-slices $\bar \cS$ such that $N_G(\bar \cS)/S_0$ is a $p'$-group. 
\item For any $\bar \cS \in \Pi_n(G)$, let $\bar \cS^{\small\to}_p$ be the unique element $\bar \cK$ of $\big(\Pi_n(G) \big)_{(p)}$, up to conjugation, such that $I_{\bar \cS,p}= I_{\bar \cK,p}$.
\item  If $\bar \cS$  is a $n$-slice of $G$, let $\bar \cS^{+}_p$ denote a slice of the form $P\bar \cS:= (PS_0, \dots, PS_n)$ of $G$, where $P$ is a Sylow $p$-subgroup of $N_G(\bar \cS)$.\\
Define inductively an increasing sequence $(\bar \cS^i )_i$ in $(\Pi_n(G), \preceq)$ by
 $\bar \cS^0= \bar \cS$ and $\bar \cS^{i+1}=(\bar \cS^i )^{+}_p $, for $i \in \NN$. We set $\bar \cS^{\infty}$ for  the largest term  of the sequence $(\bar \cS^i )_i$.

\item By Proposition~\ref{ideals}, the prime ideal of $B_n(G)$  are parametrized by pairs  $(\bar \cS, p)$ where $p$ is equal to $0$ or a prime number and  $\bar \cS \in \Pi_n(G)$ is such that  
  $\vert N_G(\bar \cS)/S_0 \vert \not \equiv 0 ~(\mod~p)$. We set  $\Theta_n (G)$  the set of  $(\bar \cS, p)$ where $p$ is equal to $0$ or a prime number and  $\bar \cS \in \Pi_n(G)$ is such that  
  $\vert N_G(\bar \cS)/S_0 \vert \not \equiv 0 ~(\mod~p)$.\\

\end{itemize}
\end{Nota}
Note that if we define a relation $  ``\overset{p}{\sim} "$ on $\Pi_n(G)$ by $\bar \cS \overset{p}{\sim} \bar \cS'$ if $\bar \cS^{\to}_p  =_G \bar \cS'^{\to}_p$ then $  ``\overset{p}{\sim} "$  is an equivalence relation. 

\begin{Prop}
The $n$-slice  $\bar \cS^{\to}_p$ is conjugate to $\bar \cS^{\infty}$.

 \end{Prop}
 
\begin{proof} 
By definition, the $n$-slice $\bar \cS^{\to}_p$ is a minimal element $\bar \cK$ of the~$(\Pi_n(G), \preceq~)$ such that
$$\phi_{\bar \cS }(\langle \bar \cK  \rangle_G):= \vert \{ g \in G/K_0~\vert~\bar \cS^g \preceq  \bar \cK \} \vert \not\equiv 0~(\mod~p).$$
Thus one can assume that $\bar \cS \preceq \bar \cK$. Since $\phi_{\bar \cS} \equiv \phi_{P\bar \cS} ~(\mod~p)$  for any $p$-subgroup $P$ of $N_G(\bar \cS)$ by Corollary~\ref{congrue}, one can also assume that $\bar \cS^{+}_p  \preceq \bar \cK$, and inductively, that $\bar \cS^{\infty}  \preceq \bar \cK$. Moreover $\phi_{\bar \cS^{\infty}} \equiv \phi_{\bar \cK } ~(\mod~p)$. As $N_G(\bar \cS^{\infty}) / S_0^{\infty}$ is a $p'$-group, it follows that $\bar \cS^{\infty} =_G \bar \cK$.
  \end{proof}
 \begin{Prop}\label{max}
   Let $(\bar \cS, p)$, $(\bar \cS', p')$ be elements of  $\Theta_n (G)$. \\ Then $I_{\bar \cS', p'} \subseteq I_{\bar \cS, p}$ if and only if
  \begin{itemize}
\item either $p'=p$ and the $n$-slices $\bar \cS'$ and $\bar \cS$ are conjugate in~$G$.
\item or  $p'=0$ and $p>0$, and the $n$-slices $\bar \cS'^{\to}_p$ and $\bar \cS$ are conjugate in $G$.
\end{itemize}
  \end{Prop}
 \begin{proof} 
 Assume that $I_{\bar \cS', p'} \subseteq I_{\bar \cS, p}$. Then there exists a surjective ring homomorphism $\phi$ from $B_n(G)/I_{\bar \cS', p'}\cong \ZZ/p'\ZZ$ to $B_n(G)/I_{\bar \cS, p}\cong \ZZ/p\ZZ$. Hence either $p'=p$ or $p'=0$ and $p>0$. If $p=p'$ then $\phi$ is a bijection, and so $I_{\bar \cS', p'}= I_{\bar \cS, p}$ which implies that $\bar \cS'$ and $\bar \cS$ are conjugate. If  $p'=0$ and $p>0$ then $\phi^G_{\bar \cS}$ is the reduction modulo $p$ of $\phi^G_{\bar \cS'}$, and so $I_{\bar \cS', p} \subseteq I_{\bar \cS, p}$. Hence  $\bar \cS'^{\to}_p$ and $\bar \cS$ are conjugate in $G$, by Corollary~\ref{ideals}. Conversely, if $\bar \cS'$ and $\bar \cS$ are conjugate then $\phi^G_{\bar \cS}=\phi^G_{\bar \cS'}$, in particular $I_{\bar \cS', 0} = I_{\bar \cS, 0}$. If $p$ is a prime then Proposition~\ref{max} implies that $\phi^G_{\bar \cS'}(x) \equiv \phi^G_{\bar \cS'^{\to}_p}(x)~{\mod ~p}$ for any $x \in B_n(G)$. So $I_{\bar \cS'^{\to}_p, 0} \subseteq I_{\bar \cS', p}$. And if $\bar \cS'^{\to}_p$ and $\bar \cS$ are conjugate in $G$, then  $I_{\bar \cS'^{\to}_p, p}=I_{\bar \cS', p} = I_{\bar \cS, p}$. So $I_{\bar \cS'^{\to}_p, 0} \subseteq I_{\bar \cS, p}$.
 \end{proof}  
 \begin{Rema} 
  \end{Rema} \begin{itemize}
\item Since for a prime $p$, the prime ideals of $\big( B_n(G) \big)_{(p)}$ are of the form $\ZZ_{(p)}I$ where $I$ is a prime ideal of $B_n(G)$ which does not meet $\ZZ-p\ZZ$, we have that $I=I_{\bar \cS, p}$ or $I=I_{\bar \cS, 0}$. Hence the connected  component  of $\big(B_n(G) \big)_{(p)} $ are indexed by $[\big(\Pi_n(G) \big)_{(p)} ]$. Moreover, the component  indexed  by the $n$-slice $\bar \cS$ consists  of a unique maximal element $\ZZ_{(p)}I_{\bar \cS, p}$ and of ideals $\ZZ_{(p)}I_{\bar \cS', 0}$, where $\bar \cS' \in \Pi_n(G)$ is such that $\bar \cS'^{\to}_p =_G\bar \cS$.
\item For $\bar \cS$ in $\big(\Pi_n(G) \big)_{(p)}$  let $\tilde e_{\bar \cS}$ denote the sum $\sum e^G_{\bar \cS'}$ in $\QQ B_n(G)$ where $\bar \cS'$ ranges over all $n$-slices in $[\Pi_n(G)]$  such that $\bar \cS$ is conjugate to $\bar \cS'^{\to}_p$. Then the distinct
$\tilde e_{\bar \cS}$ are   the primitive idempotents of $\big( B_n(G) \big)_{(p)}$. This is because, for any commutative ring $R$ the connected component of $Spec R$ corresponding to a primitive idempotent $e$ in $R$ consists of all prime ideals of $R$ which contain $1-e$. If $e$ is a primitive idempotent of $\big( B_n(G) \big)_{(p)}$ corresponding to the connected component of $\big( B_n(G) \big)_{(p)}$ indexed by $\bar \cS, p$, then $1-e \in \ZZ_{(p)}I_{\bar \cS', 0}$ if and only if $\bar \cS'^{\to}_p=_G \bar \cS$.
\end{itemize}

 \begin{Prop} 
Two ideals    $ I_{\bar \cS, p}$   and $I_{\bar \cS', p'} $  are in the same connected component of $Spec(B_n(G))$ if and only $D^{\infty}(\bar \cS)$ is conjugate to $ D^{\infty}(\bar \cS')$, where \\
$D^{\infty}(\bar \cK):= (D^{\infty}(\bar \cK_0), \dots, D^{\infty}(\bar \cK_n))$ and $D^{\infty}(\bar \cK_i)$ denotes the last term in the derived series of $\bar \cK_i$.
In particular, $Spec(B_n(G))$ is connected  if and only if $G$ is solvable.

 \end{Prop}  
 \begin{proof} 
Recall that if $R$ is a Noetherian ring. For any prime ideal $I \in Spec(R)$, let $\overline{I}= \{ P~\vert~ P \in Spec(R), ~P \supset I\}$ be the closure of $I$ in $Spec(R)$. Then, two ideals $P$ and $P'$ are in the same connected component of $Spec(R)$ if and only if there exists a series of minimal ideals $I_1, \dots, I_n$ with $P \in \overline{I_1}$,  $P' \in \overline{I_n}$ and $\overline{I_i }\cap \overline{I_{i+1}} \neq \emptyset$ for $i=1, \dots,n-1$. Now, if $R= B_n(G)$, then $\bar I_{\bar \cS, 0} \cap \bar I_{\bar \cS', 0} \neq \emptyset$ if and only if the $n$-slice $\bar \cS^{\to}_p$ is conjugate to $\bar \cS'^{\to}_p$, for some prime $p$. Hence, if  $I_{\bar \cS, p} $ and $I_{\bar \cS', p'} $ are in the same connected component of $Spec(B_n(G))$ then $ D^{\infty}(\bar \cS)$ is conjugate to $ D^{\infty}(\bar \cS')$.
The ideals $I_{\bar \cS, p}$ and $I_{D^{\infty}(\bar \cS), 0}$ are in the same connected component. Indeed,  one can find a series of normal subgroups of $S_0$  such that $D^{\infty}(S_0) = S_0^{(n)} \triangleleft S_0^{(n-1)}\triangleleft \dots \triangleleft S_0^{(1)} \triangleleft  S_0^{(0)}= S_0$ such that $S_0^{(i-1)}/ S_0^{(i)}$ is a $p_i$-group for some prime $p_i$ ( $i=1,\dots, n$). Hence letting,
$\bar \cK_i = (S_0^{(n)}, S_1, \dots, S_n)$ 
 for $i=0,\dots,n$ one obtains
$$\bar \cS=\bar \cK_0 \overset{p_0}{\sim} \bar \cK_1\overset{p_1}{\sim} \dots \overset{p_{n-2}}{\sim}      \bar \cK_{n-1}    \overset{p_{n-1}}{\sim}   \bar \cK_n= (D^{\infty}(S_0), S_1, \dots, S_n),$$ and 

$I_{\bar \cS, p} \in  \bar I_{\bar \cK_0, 0}$,  $I_{D^{\infty}(S_0), S_1, \dots, S_n, 0} \in  \bar I_{\bar \cK_n, 0}$ and  $\bar I_{\bar \cK_{i-1}, 0} \cap \bar I_{\bar \cK_{i}, 0} \neq \emptyset$. So, $I_{\bar \cS, p}$  and $I_{D^{\infty}(S_0), S_1, \dots, S_n, 0}$ are in the same connected component. By the same proceed, one prove that\\ $I_{(D^{\infty}(S_0), D^{\infty}(S_1), \dots, S_n), 0}$ and $I_{\bar \cS, p}$ are in the same component, and so one.
 \end{proof}

\section{Green biset functor }

Let $R$ be a commutative ring with identity.
The {\em biset category} over $R$ will be denoted by $R \cC$: its objects are all finite groups, and that for finite groups $G$ and H, the hom-set $\Hom_{R\cC}(G,H)$ is $RB(H,G)= R \otimes B(H,G)$, where $B(H,G)$ is the Grothendieck group of the category of finite $(H,G)$-bisets. The composition of morphisms in $R \cC$ is induced by $R$-bilinearity from the composition of bisets (see \cite{S.Bouc}  Definition 3.1.1). \par
A good choice of a family  $\cG$ of finite groups and for every $G,H \in \cG$, a set $\Gamma(G,H)$ of subgroups of $G \times H$ can  lead to an important category. For example  if we  fix a non-empty class $\cD$ of finite groups closed under subquotients and cartesian products and  we denote by $R \cD$ the full subcategory of $R \cC$  consisting of groups in $\cD$, then $R \cD$ is a replete subcategory of $R \cC$ in the sense of Bouc \cite{S.Bouc}.\par
The category of {\em biset functors}, i.e. the category of $R$-linear functors from $R\cC$  to the category $R-\Mod$ of all $R$-modules, will be denoted by $\cF_{R}$. The category $\cF_{\cD,R}$ of $\cD$-biset is the category of $R$-linear functors from $R\cD$ to $R-\Mod$.\par

Let $G$ be a finite group,  $H$ a subgroup of of $G$ and $N$ be a normal subgroup of $G$. One sets 
\begin{itemize}
\item 
$\Res_H^G := [{}_HG_G]$ for the image in the {\em biset Burnside group} $B(H,G)$ of the isomorphism class of $G$ where $G$ is viewed as an $(H,G)$-biset via the left and right multiplication.
\item $\Ind_H^G := [{}_GG_H]$ for the image in  $B(H,G)$ of the isomorphism class of $G$ where $G$ is viewed as a $(G,H)$-biset via the left and right multiplication.
\item $\Inf^G_{G/N} := [{}_G(G/N)_{G/N}]$ for the image in  $B(G,G/N)$ of the isomorphism class of $G/N$ where $G/N$ is viewed as a $(G,G/N)$-biset via the the canonical epimorphism $G \ra G/N$ and left and right multiplication.
\item $\Dof^G_{G/N} := [{}_{G/N}(G/N)_{G}]$ for the image in  $B(G/N,G)$ of the isomorphism class of $G/N$ where $G/N$ is viewed as a $(G,G/N)$-biset via the the canonical epimorphism $G \ra G/N$ and left and right multiplication.
\item If $f:G\ra H$ is a group isomorphism, then we set $\Iso(f):= [{}_{H}H_{G}]$ for the image in  $B(H,G)$ of the isomorphism class of $H$ where $H$ is considered  as an $(H,G)$-biset via $hxg=hxf(g)$ for $h,x \in H$ and $g \in G$.
\end{itemize} 
It is possible to verify that all elements in $B(H,G)$ are sums of $[H\times G/L]$ where $L$ runs through subgroups of $H\times G$ and  that these satisfy the following decomposition:
\begin{equation}
[H\times G/L]= \Ind^H_D \circ \Inf^D_{D/C} \circ \Iso (f) \circ  \Dof^B_{B/A}\circ \Res^G_B 
\end{equation}
where $(D,C)$ and $(B,A)$ are $1$-slices of $H$ and $G$ respectively with the additional properties that $C \lhd D$ and $A \lhd B$, and $f$ is some the group isomorphism from $B/A$ to $D/C$ (see \cite{S.Bouc} Lemma 2.3.26 for more details).

A Green $\cD$-biset functor is defined as a monoid in $\cF_{\cD,R}$. This is equivalent to the following definitions:
\begin{Defi} [\cite{S.Bouc} Definition 8.5.1] \label{Green}
A $\cD$-biset functor $A$ is a Green $\cD$-biset  functor if it is equipped with a linear products $A(G) \times A(H) \ra A(G \times H)$ denoted by $(a,b) \mapsto a \times b$, for groups, $G$, $H$ in $\cD$, and an element $\epsilon_A \in A(1)$, satisfying the following conditions:
\begin{enumerate}
\item (Associativity). Let $G,H$ and $K$ be groups in $\cD$. If we consider the canonical isomorphism from $G \times (H \times K)$ to $(G \times H) \times K$, then for any $a \in A(G)$, $b \in A(H)$ and $c \in A(K)$

$$(a\times b) \times c = A\bigg(\Iso^{(G \times H) \times K}_{G \times (H \times K)}\bigg) \big(a \times (b \times c) \big).$$
\item (Identity element).  Let $G$ be a group in $\cD$ and consider the canonical isomorphisms $1 \times G \ra G$ and $G \times 1 \ra G$. Then for any $a \in A(G)$
$$a = A\big(\Iso^G_{1 \times G}\big)( \epsilon_A  \times a)=  A\big(\Iso^G_{ G\times 1}\big)( a \times \epsilon_A )$$
\item (Functoriality). If $\phi: G \ra G'$ and $\psi: H \ra H'$ are morphisms in $R \cD$, then for any $a \in A(G)$ and $b \in A(H)$
$$A \big(\phi \times \psi\big)(a \times b)= A(\phi)(a) \times A(\psi)(b).$$

\end{enumerate}
 
\end{Defi} 
 There is an equivalent way of defining a Green biset functor given by Romero in (\cite{N.Romero3} Lema 4.2.3):

\begin{Defi}[\cite{N.Romero3} Definiciòn 3.2.7]
A  $\cD$-Green biset functor is an object $A \in \cF_{\cD,R}$ together with the datum of an $R$-algebra sructure on each $A(H)$, $H \in \cD$, such that the following axioms are satisfied for all groups in $K$ and $G$ in $\cD$ and all group homomorphisms $K \ra G$: 
\begin{enumerate}
\item For the $(K,G)$-biset $G$, which we denote by $G_r$,  the morphism $A(G_r)$ is a ring homomorphism.
\item For the $(G,K)$-biset $G$, denoted by $G_l$, the morphism $A(G_l)$ satisfies the Frobenius identities for all $b \in A(G)$ and $a \in A(K)$,
\begin{eqnarray*}
A(G_l)(a) \cdot b&=& A(G_l)\big(a \cdot A(G_r)(b) \big)\\
b\cdot A(G_l)(a)&=& A(G_l)\big(A(G_r)(b) \cdot a \big)
\end{eqnarray*}
where $\cdot$ denotes the ring product on $A(G)$, resp. $A(K)$.

\end{enumerate}
\end{Defi}
\begin{Defi}
If $A$ and $C$ are Green $\cD$-biset functors, a morphism of Green $\cD$-biset functors from $A$ to $C$ is a natural transformations $f: A \ra C$ such that $f_{H \times K}(a \times b)= f_H(a) \times f_K(b)$ for any groups $H$ and $K$ in $\cD$ and any $a \in A(H)$, $b \in A(K)$, and such that $f_1(\epsilon_A)=\epsilon_C$.
\end{Defi}

\begin{Prop}

The correspondence $$G \mapsto B_n(G)$$ defines a structure of Green biset functor.
\end{Prop} 

\begin{proof} There are several steps:
\begin{itemize}
\item  Any $(n,G)$ simplex $\cX^f_n$ give rise a $(n, H)$-simplex $U \times_G \cX^f_n$ in the following rule (following the notion of Definition 2.3.11 in \cite{S.Bouc}):
$$U \times_G \cX^f_n: (\xymatrix{U\times_GX_0 \ar[r]^{Uf_1} &  U\times_GX_{1}\ar[r]^-{Uf_{2}} & \dots U\times_GX_{n-1} \ar[r]^-{Uf_{n}} &  U\times_GX_{n}}) $$
defined by  $Uf_i: U\times_GX_{i-1} \ra U\times_GX_i, (u,x) \mapsto (u, f_i(x)) $ for any $i= 1,\dots, n$.\\
Let $\mu= (\mu_i)$ be a morphism from $\cX^f_n$ to $\cY^g_n$. Then $\mu$ defines a morphism of $(n,H)$-simplices  from $U \times_G \cX^f_n$ to $U \times_G \cY^g_n$ given by  $$U\times_G\mu= ( U\mu_i: (u,x)\mapsto (u, \mu_i(x)))_i.$$ Indeed, for any $i$, $x_i \in X_i$ and $u \in U$, we have 
\begin{eqnarray*}
(U\mu_i)\circ (Uf_i)(u,x_{i-1}) &=& (u, \mu_if_i(x_{i-1})) \\
&=&(u, g_{i}\mu_{i-1}(x_{i-1}))\\
&=& (Ug_i)\circ (U\mu_{i-1})(u,x_{i-1})
\end{eqnarray*}
Therefore, the correspondence $I_U: \cX^f_n \mapsto U \times_G \cX^f_n $ is a functor from the  $(n,G)$-simplices to $(n,H)$-simplices. On the other hand,  it is straightforward to show that the defining relations of $B_n(G)$ are mapped to the defining relations of  $B_n(H)$. Hence, the later functor induces a homomorphism of groups $B_n(U): B_n(G) \ra B_n(H)$.
\item The  correspondence $G \mapsto B_n(G)$ defines a structure of  biset functor:\\
Clearly, if $U \cong U'$ (as (H,G)-bisets) then the functors $I_U$ and $I_{U'}$ are isomorphic. So, $B_n(U)=B_n(U')$. If $U$ has the form $U=U_1 \sqcup U_2$ (as (H,G)-bisets) then $I_U= I_{U_1} \sqcup I_{U_2}$, and so $B_n(U)= B_n(U_1) + B_n(U_2)$. We may state that for any $(K,H)$-biset $V$, we have a isomorphism of $(n,K)$-simplices  between $ V \times_H(U \times \cX^f_n)$ and $ (V \times_HU) \times \cX^f_n$ which induces an isomorphism $I_V \circ I_U \cong I_{V \times_HU}$ and so $B_n(V) \circ B_n(U)= B_n(V \times_HU)$. Finally, since $I_{\Id_G} \cong 1$(the functor identity, we have $B_n(Id_G)=1_{B_n(G)}$. This shows that we have a functor from the biset category to the category $\ZZ-\Mod$.
\item Let  $\cY^g_n$ be a $(n,H)$-simplex. Then the product $\cX^f_n \times \cY^g_n$ is a $(n, G\times H)$-simplex in the obvious way. It induces a well-defined bilinear $$\times: B_n(G) \times B_n(H) \ra B_n(G\times H), ~~((a,b) \mapsto a \times b),$$
and $(n,1)$-simplex $\textbf e$ is the identity of the product, up to identification $G\times 1=G$. One verifies easily that the axioms of Definition~\ref{Green} are satisfied. 

\end{itemize}
\end{proof}
\begin{Prop}
The functors $d_j $, $j=1,\dots, n$   (resp.  $s_i $, $i=0,\dots, n-1$)
induce morphisms of Green biset functors  $$d_j : B_n \ra B_{n-1} ~~~~\big(\text{resp.}~   s_i: B_{n-1} \ra B_{n} \big)$$ such that the identities (\ref{1}), (\ref{2}), (\ref{3}) hold.
\end{Prop} 
\begin{proof}
This is a simple verification. 
\end{proof}

\section*{Acknowledgement} The author would like to thank to acknowledge support from CCM-UNAM-Morelia.

 \bibliographystyle{plain}
\bibliography{mabiblio3}

${}$\\
${}$\\
{\small
Ibrahima Tounkara\\
Centro de Ciencias Matem\'aticas\\
UNAM,\\
C.P. 58089\\
Morelia Mich\\
Mexico\\
e-mail: tounkara@matmor.unam.mx}

\end{document}